\documentclass[12pt,reqno]{amsart}

\usepackage{amssymb} 
\usepackage{dsfont} 
\usepackage{mathrsfs} 
\usepackage{stmaryrd} 
\usepackage{xcolor}

\usepackage{amsmath,amssymb,amsfonts,amsthm}
\usepackage{mathrsfs}
\usepackage{mathtools}
\usepackage{aliascnt}

\usepackage[utf8]{inputenc}
\usepackage[T1]{fontenc}
\usepackage{lmodern}
\usepackage[english]{babel}
\usepackage{microtype}
\usepackage{csquotes}

\usepackage{tikz}
\usetikzlibrary{math}

\definecolor{midnightblue}{rgb}{0.1, 0.1, 0.44}
\usepackage[colorlinks,allcolors=midnightblue,pdfusetitle]{hyperref}
\AtBeginDocument{}
\AtBeginDocument{}
\AtBeginDocument{}
\hypersetup{
	pdfsubject={},
	pdfkeywords={}}

\usepackage{footmisc}

\usepackage{enumitem}
\setlist[enumerate,1]{label=(\roman*)}

\usepackage{booktabs}

\DeclareMathAlphabet{\mathpzc}{OT1}{pzc}{m}{it} 

\newtheorem{Thm}{Theorem}[section]

\newtheorem{Lem}{Lemma}[section]
\newtheorem{Prop}{Proposition}[section]

\newtheorem{Def}{Definition}[section]
\newtheorem{Pb}{Problem}[section]

\theoremstyle{definition}

\theoremstyle{definition}

\makeatletter
\@addtoreset{equation}{section}
\makeatother

\newcommand\void{\varnothing} 
\newcommand\setmeno{\!\smallsetminus\!} 
\newcommand\parti{\mathscr{P}} 
\newcommand\function{\longrightarrow} 
\newcommand\en{\mathbb{N}} 
\newcommand\ar{\mathbb{R}} 
\newcommand{\eps}{\varepsilon} 
\providecommand{\clint}[1]{\hspace{0.045ex}\left[#1\right]} 
\providecommand{\opint}[1]{\hspace{0.15ex}\left]#1\right[\hspace{0.15ex}} 
\providecommand{\clsxint}[1]{\hspace{0.1ex}\left[#1\right[\hspace{0.15ex}} 
\providecommand{\cldxint}[1]{\hspace{0.15ex}\left]#1\right]} 

\DeclareMathOperator{\Int}{int} 

\newcommand{\borel}{\mathscr{B}} 
\DeclareMathOperator{\de}{d \! \hspace{0.2ex}} 
\newcommand\leb{\mathpzc{L}} 
\newcommand\vartot[1]{\!\left\bracevert\! #1 \!\right\bracevert\!} 

\renewcommand\H{\mathcal{H}} 
\newcommand\duality[2]{\langle #1,#2 \rangle} 
\newcommand\lduality[2]{\left\langle #1,#2 \right\rangle} 
\newcommand\norm[2]{\Vert #1\Vert_{#2}} 
\newcommand\lnorm[2]{\left\Vert #1\right\Vert_{#2}} 
\renewcommand\d{\textsl{d}} 
\DeclareMathOperator{\Proj}{Proj} 
\renewcommand{\P}{\mathrm{P}} 
\newcommand{\Conv}{\mathscr{C}} 
\newcommand\A{\mathcal{A}} 
\newcommand\B{\mathcal{B}} 
\newcommand\Z{\mathcal{Z}} 
\newcommand\K{\mathcal{K}} 

\newcommand\e{\textsl{e}} 

\DeclareMathOperator{\V}{V} 
\newcommand{\BV}{{\textsl{BV}\hspace{0.17ex}}} 
\newcommand{\Czero}{{\textsl{C}\hspace{0.18ex}}} 
\renewcommand{\L}{{\textsl{L}\hspace{0.17ex}}} 
\DeclareMathOperator{\D}{D\!} 

\newcommand{\convergedebstar}{\stackrel{*}{\rightharpoonup}}

\newcommand\indicator{\mathds{1}} 

\newcommand{\Step}{{\textsl{St}\hspace{0.17ex}}} 

\newcommand\hausd{\mathpzc{H}} 
\newcommand\C{\mathcal{C}} 

\DeclareMathOperator{\pV}{V} 

\definecolor{blu}{rgb}{0.1,0.1,1}

\definecolor{green}{rgb}{0.0, 0.5, 0.0}

\definecolor{marr}{rgb}{0.63, 0.47, 0.35}

\begin{document}


\title[Sweeping processes]{Excess-continuous prox-regular sweeping processes}

\author{Vincenzo Recupero, Federico Stra}
\thanks{Vincenzo Recupero is partially supported by the GNAMPA-INdAM project 2025 ``Analisi e controllo di modelli evolutivi con fenomeni non locali''.}

\date{July 28, 2025}

\dedicatory{Dedicated to Prof.\ Alexander Plakhov on the occasion of his 65\textsuperscript{th} birthday.}

\address{
	\textbf{Vincenzo Recupero} \\
	Dipartimento di Scienze Matematiche \\
	Politecnico di Torino \\
	C.so Duca degli Abruzzi 24 \\
	I-10129 Torino \\
	Italy. \newline
	{\rm E-mail address:}
	{\tt vincenzo.recupero@polito.it}  }
\address{
	\textbf{Federico Stra} \\
	Dipartimento di Scienze Matematiche \\
	Politecnico di Torino \\
	C.so Duca degli Abruzzi 24 \\
	I-10129 Torino \\
	Italy. \newline
	{\rm E-mail address:}
	{\tt federico.stra@polito.it}
}




\begin{abstract}
	In this paper we consider the Moreau's sweeping processes driven by a time dependent prox-regular
	set $\C(t)$ which is continuous in time with respect to the asymmetric distance $\e$ called
	\emph{the excess},  defined by $\e(\A,\B) := \sup_{x \in \A} \d(x,\B)$ for every pair of sets
	$\A$, $\B$ in a Hilbert space. As observed by J.J. Moreau in his pioneering works, the excess
	provides the natural topological framework for sweeping process. Assuming a uniform interior cone
	condition for $\C(t)$, we prove that the associated sweeping process has a unique solution,
	thereby improving the existing result on continuous prox-regular sweeping processes in two
	directions: indeed, in the previous literature $\C(t)$ was supposed to be continuous in time with
	respect to the symmetric Hausdorff distance instead of the excess and also its boundary
	$\partial\C(t)$ was required to be continuous in time, an assumption which we completely drop.
	Therefore our result allows to consider a much wider class of continuously moving constraints.
\end{abstract}


\maketitle

\tableofcontents

\thispagestyle{empty}


\section{Introduction}

\emph{Sweeping processes} are a class of evolution problems with unilateral constraints which were originally introduced in \cite{Mor71,Mor72} by J.J. Moreau, motivated by problems in elastoplasticity and nonsmooth mechanics (cf., e.g., \cite{Mor74,Mor76b,Mor02}, and the monograph \cite{Mon93}). Later sweeping processes have then found applications in several diverse disciplines: in economic theory (for instance in \cite{Hen73, Cor83, FlaHirJou09}), in electrical circuits (see, e.g.,
\cite{AddAdlBroGoe07,BroThi10, AcaBonBro11, AdlHadThi14}), in crowd motion modeling (cf., e.g.,
\cite{MauVen08, MauRouSan10, MauVen11, MauRouSanVen11, MauRouSan14, DimMauSan16, CaoMorl17}), and in other fields (see, e.g., papers \cite{Thi16,NacThi20,KreMonRec23,KamPetRec24} and their references).

Sweeping processes can be described as follows. Fix a final time $T > 0$, let $\H$ be a real Hilbert space, and let $\C(t)$ be a given nonempty, closed subset of $\H$ depending on time $t \in \clint{0,T}$. One has to find a function $y : \clint{0,T} \function \H$ such that
\begin{alignat}{3}
	 & y(t) \in \C(t)             & \quad                           & \forall t \in \clint{0,T}, \label{y in C - Lip - intro}                           \\
	 & y'(t) \in -N_{\C(t)}(y(t)) & \quad                           & \text{for $\leb^{1}$-a.e. $t \in \clint{0,T}$}, \label{diff. incl. - Lip - intro} \\
	 & y(0) = y_{0},              & \label{in. cond. - Lip - intro}
\end{alignat}
$y_0$ being a prescribed point in $\C(0)$. Here $y'(t)$ is the derivative of $y$ at $t$, $\leb^1$ is the Lebesgue measure, and $N_{\C(t)}(y(t))$ is the (proximal) exterior normal cone to $\C(t)$ at $y(t)$ (all the precise definitions needed in the paper will be provided in the next section).
In \cite{Mor71} Moreau found a unique Lipschitz continuous function $y$ solving \eqref{y in C - Lip - intro}--\eqref{in. cond. - Lip - intro}, by assuming that the sets
$\C(t)$ are convex, and the mapping $t \longmapsto \C(t)$ is Lipschitz continuous in time, when the family of closed subsets of $\H$ is endowed with the Hausdorff metric
\[
	\d_\hausd(\A,\B) := \max\left\{\sup_{x \in \A} \d(x,\B),\sup_{y \in \B} \d(x,\A)\right\},
\]
with $\A$, $\B$ nonempty closed sets in $\H$.

In the celebrated paper \cite{Mor77}, the formulation \eqref{y in C - Lip - intro}--\eqref{in. cond. - Lip - intro} was then extended to the case when the mapping $t \longmapsto \C(t)$ is of bounded variation and right continuous with respect to $\d_\hausd$: it is proved that there is a unique right continuous $y \in \BV(\clint{0,T};\H)$, the space of right continuous $\H$-valued functions with bounded variation, such that there exist a positive measure $\mu$ and a $\mu$-integrable function $v : \clsxint{0,\infty} \function \H$ satisfying
\eqref{y in C - Lip - intro} and \eqref{in. cond. - Lip - intro} together with
\begin{alignat}{3}
	 & \! \D y = v \mu,          & \label{Dy = w mu - intro}                                                                               \\
	 & v(t) \in -N_{\C(t)}(y(t)) & \quad                     & \text{for $\mu$-a.e. $t \in \clint{0,T}$}, \label{diff. incl. - BV - intro}
\end{alignat}
where $\D y$ denotes distributional derivative of $y$, which is a measure since $y$ is of bounded variation.

Actually in \cite{Mor77} the moving set $\C$ is assumed to be only with \emph{bounded retraction}, rather than with bounded variation, i.e. the Hausdorff distance is replaced by the asymmetric distance
\[
	\e(\A,\B) := \sup_{x \in \A} \d(x,\B),
\]
called \emph{excess of $\A$ over $\B$}, with $\A$, $\B$ nonempty closed sets in $\H$. Indeed, as Moreau observed, ``some
regularity for the motion of the set $\C(t)$ is needed only when this set retracts, effectively sweeping the point $y(t)$'', so that
it is natural assume only ``some unilateral rating for the displacement of sets''.

After the pioneering work of Moreau, convex sweeping processes in the $BV$ framework have received a great deal of attention, but most of the papers addressed theorems involving the Hausdorff distance $\d_\hausd$ rather than the excess $\e$, probably due to the fact that the asymmetric excess introduces certain technical difficulties (see \cite{Mor74a,Mor74b}). Several developments can be found, e.g., in the monograph \cite{Mon93}, in papers
\cite{Cas73, Cas76, Cas83, CasDucVal93, BroKreSch04, KreRoc11, Thi16, DimMauSan16, KopRec16, RecSan18,Rec25}, and in the references therein. These researches can be generalized in two directions. On one hand one can deal with driving sets which are non-convex, and on the other hand one can relax the assumption that $\C(t)$ has bounded variation.

The analysis of
non-convex sweeping processes started in \cite{Val88} and, since then, has called the attention of many other authors, e.\,g., \cite{CasMon96, ColGon99, Ben00, BouThi05}. Concerning existence and uniqueness results in the non-convex setting, we mention, for instance,
\cite{Thi03, EdmThi06, CheMon07, BerVen12, SenThi14, BerVen15, AdlNacThi17, NacThi20, KreMonRec21, KreMonRec22,
	KreMonRec23, KopRec23, RecStr25}.
More precisely these papers consider the family of \emph{uniformly prox-regular} subsets of $\H$, which can be described as those closed subsets $\K$ such that there exists $r > 0$ for which the function
\[
	x \longmapsto \d_\K(x) := \inf\{\norm{z-x}{}\ :\ z \in \K\}
\]
is differentiable on the open set $U_r(\K) := \{x \in \H\ :\ 0 < d_\K(x) < r\}$. A fundamental property of uniform prox-regular sets is that any point in $U_r(\K)$ admits a unique projection in $\K$, therefore, such sets
can be regarded as a natural generalization of convex sets. Uniform prox-regular sets were introduced in \cite{Fed59} under the name \emph{positively reached sets} in the finite dimensional case. Further properties of these sets were investigated in \cite{Via83} (therein called \emph{weak convex sets}), and the notion of prox-regularity was later extended to infinite dimensional spaces in \cite{ClaSteWol95, PolRocThi00, ColMon03}.

The second direction of research about processes driven by convex sets with unbounded variation was instead addressed independently in \cite{Cas83} and in Section 19 (mostly written by A. Vladimirov) of \cite{KraPok83}.
In such works the moving constraint is required to have non-empty interior: this assures that the solution has bounded variation and that formulation \eqref{diff. incl. - BV - intro} makes sense.
For other contributions see, e.g., \cite{Mon93, CasMon95, KreLau02} and their references.

The first paper where the two directions are addressed simultaneously, namely when the moving non-convex constraint has unbounded variation, was \cite{ColMon03} where it is assumed that $\C(t)$ is continuous in time with respect to the Hausdorff distance $\d_\hausd$. In \cite{ColMon03} another geometric condition is also required, namely $\C(t)$ has uniform interior cone condition, which essentially means that cusps are not admitted on the boundary, thereby assuring the bounded variation of the solution.
A further step into this research can be found in the recent paper \cite{KreMonRec22}, where $\C(t)$ has uniform interior cone condition, but with special form $\C(t) = u(t) - \Z(w(t))$ for given time-dependent $u(t) \in \H$ and parameter $w(t)$ which are assumed to be regulated right-continuous functions, i.e., they admit the one-sided limits $u(t+)= u(t), u(t-), w(t+)=w(t), w(t-)$ at every point $t \in [0,T]$.

As we mentioned above, most papers about prox-regular sweeping processes address theorems where the rating of $\C(t)$ is prescribed with respect to the Hausdorff distance, but not to the excess. To the best of our knowledge, in the prox-regular setting, the excess is addresses only in \cite{Thi16} and in \cite{Rec25} where $\C(t)$ is assumed to be of bounded retraction: in \cite{Thi16} the uniqueness of solution is proved, while the existence is achieved in \cite{Rec25}.

The aim of the present paper is instead to study the case when $\C(t)$ move continuously, and to improve the result of \cite{ColMon03} in two directions. First, we consider a general $\C(t)$ having uniform interior cone condition, but we assume its continuity in time only with respect to the excess, i.e. we assume that the following excess-continuity modulus
\[
	\omega(\delta) := \sup\{\e(\C(s),\C(t))\ :\ 0 \le t - s < \delta\}
\]
converges to zero as $\delta \searrow 0$. Second, we remove the assumption made in \cite{ColMon03} that
$t \longmapsto \partial\C(t)$ is continuous w.r.t.\ the Hausdorff distance: we do not require any kind continuity
of $\partial \C(t)$, thereby enlarging the class of admissible sets in a meaningful way.

The paper is organized as follows. In the next section we present some preliminaries and in Section
\ref{Main res} we state the main results of the paper. In Section \ref{proofs} we perform all the necessary proofs. Section \ref{proofs} is divided in several subsections.


\section{Preliminaries and notations}\label{S:prelim}

In this section we recall the main definitions and tools needed in the paper. The set of integers greater than or equal to $1$ will be denoted by $\en$. For every real number $t \in \ar$, we define $\lceil t \rceil := \min\{k \in \mathbb{Z}\ :\ t \le k\}$.

\subsection{Prox-regular sets}

We recall here some notions of non-smooth analysis. We refer the reader to the monographs
\cite{RocWet98, ClaLedSteWol98, Thi23a, Thi23b}. Throughout this paper we assume that
\begin{equation}\label{H-prel}
	\begin{cases}
		\text{$\H$ is a real Hilbert space with inner product $(x,y) \longmapsto \duality{x}{y}$}, \\
		\norm{x}{} := \duality{x}{x}^{1/2}, \quad x \in \H,
	\end{cases}
\end{equation}
and we endow $\H$ with the natural metric defined by $(x,y) \longmapsto \norm{x-y}{}$, $x, y \in \H$.
If $\rho > 0$ and $x \in \H$ we set $B_\rho(x) := \{y \in \H\ :\ \norm{y-x}{} < \rho\}$ and
$\overline{B}_\rho(x) :=  \{y \in \H\ :\ \norm{y-x}{} \le \rho\}$. If $\Z \subseteq \H$, the closure, the boundary and the interior of $\Z$ will be respectively denoted by $\overline{\Z}$, $\partial \Z$, and $\Int(\Z)$. If $x \in \H$ we also set
$\d_\Z(x) := \d(x,\Z) := \inf_{s \in \Z} \norm{x-s}{}$, defining in this way a function $\d_\Z : \H \function \mathbb{R}$.

\begin{Def}
	If $\Z \subseteq \H$, $\Z \neq \void$, and $y \in \H$ then we set
	\begin{equation}
		\Proj_\Z(y) := \left\{x \in \Z\ :\ \norm{x-y}{} = \inf_{z \in \Z} \norm{z-y}{}\right\}.
	\end{equation}
\end{Def}

Statement (2.3) below can be found, e.g., in \cite[Proposition 1.3]{ClaLedSteWol98}. In fact, given
$y \in \mathcal{H}$ and $x \in \Z$, we see that $\|x-y\|^2 \leq \|z-y\|^2$ (resp. $\|x-y\|^2 < \|z-y\|^2$) for all $z \in \Z \setmeno \{x\}$ if and only if
$x \in \Proj_\Z(y)$ (resp. $\{x\} = \Proj_\Z(y)$), and writing
\[
	\|z-y\|^2 = \|z-x\|^2 - 2\langle y-x,z-x\rangle + \|x-y\|^2
\]
we then have the following

\begin{Prop}
	If $\Z \subseteq \H$, $\Z \neq \void$, and $y \in \H$, then
	\begin{equation}
		x \in \Proj_\Z(y) \ \Longleftrightarrow \
		\left[\ x \in \Z,\quad  \duality{y-x}{z-x} \le \frac{1}{2}\norm{z-x}{}^2 \ \ \forall z \in \Z\ \right] \label{ch-pr}
	\end{equation}
	and
	\begin{equation}
		\{x\} = \Proj_\Z(y) \quad \Longleftrightarrow \
		\left[\ x \in \Z,\quad \duality{y-x}{z-x} < \frac{1}{2}\norm{z-x}{}^2 \ \ \forall z \in \Z \setmeno \{x\}\ \right] \label{ch-pr-str}
	\end{equation}
\end{Prop}

If $\K$ is a nonempty closed subset of $\H$ and $x \in \K$, then $N_\K(x)$ denotes the \emph{(exterior proximal) normal cone of $\K$ at $x$} which is defined by setting
\begin{equation}\label{normal cone}
	N_\K(x) := \{\lambda(y-x) \ :\ x \in \Proj_{\K}(y),\ y \in \H,\ \lambda \ge 0\}.
\end{equation}

Let us also recall the following proposition (see, e.g., \cite[Proposition 1.5]{ClaLedSteWol98}).

\begin{Prop}\label{propsigma}
	Assume that $\K$ is a closed subset of $\H$ and that $x \in \K$. We have that $u \in N_\K(x)$ if and only if there exists $\sigma \ge 0$ such that
	\begin{equation}\label{prox normal ineq}
		\duality{u}{z-x} \le \sigma \norm{z-x}{}^2 \qquad \forall z \in \K.
	\end{equation}
\end{Prop}

Now we recall the notion of \emph{uniformly prox-regular set}, introduced in
\cite[Section 4, Theorem 4.1-(d)]{ClaSteWol95}) under the name of \emph{proximal smooth set}, and generalized in \cite[Definition 1.1, Definition 2.4, Theorem 4.1]{PolRocThi00}.

\begin{Def}
	If $\K$ is a nonempty closed subset of $\H$ and if $r \in \opint{0,\infty}$, then we say that $\K$ is
	\emph{$r$-prox-regular} if for every $y$ the following implication holds true:
	\begin{equation}
		0 < \d_{\K}(y) < r
		\ \Longrightarrow\
		\begin{cases}
			\Proj_\K(y) \neq \void, \\
			x \in \Proj_\K\left(x+r\dfrac{y-x}{\norm{y-x}{}}\right) \qquad \forall x \in \Proj_\K(y).
		\end{cases}
	\end{equation}
	The family of $r$-prox-regular subsets of $\H$ will be denoted by $\Conv_r(\H)$. We will also indicate by
	$\Conv_0(\H)$ the family of nonempty closed subsets of $\H$, and by $\Conv_\infty(\H)$ the family of nonempty closed convex subsets of $\H$.
\end{Def}

It is useful to take into account the following easy property.

\begin{Prop}\label{1proj in seg}
	Let $\K$ be a closed subset of $\H$. If $y \in \H \setmeno \K$ and $x \in \Proj_\K(y)$ then $\Proj_\K(x + t(y-x)) = \{x\}$ for every $t \in \opint{0,1}$. Hence if $\K$ is also $r$-prox-regular for some $r > 0$ it follows that $\Proj_\K(u)$ is a singleton for every
	$u \in \H$ such that $\d_{\K}(u) < r$.
\end{Prop}

\begin{proof}
	If $t \in \opint{0,1}$ then, thanks to \eqref{ch-pr}, for every $z \in \K \setmeno \{x\}$ we have that
	\[
		\duality{x + t(y-x) - x}{z - x} = t\duality{y-x}{z-x} \le \frac{t}{2}\norm{z-x}{}^2 < \frac{1}{2}\norm{z-x}{}^2,
	\]
	therefore $\Proj_\K(x + t(y-x)) = \{x\}$ by virtue of \eqref{ch-pr-str} and the first statement is proved. If $\K$ is also $r$-prox-regular and $\d_{\K}(u) < r$, pick $x \in \Proj_\K(u)$ and let $y = x + r(u-x)/\norm{u-x}{}$. Then
	$x \in \Proj_{K}(y)$, $u = x + t_u(y-x)$ with $t_u = r/\norm{u-x}{} \in \opint{0,1}$, so that $\{x\} = \Proj_\K(u)$.
\end{proof}

\begin{Def}
	If $r \in \opint{0,\infty}$ and if $\K \subseteq \H$ is $r$-prox-regular, then we can define the function
	$\P_\K : \{y \in \H\ :\ \d_\K(y) < r\} \function \K$ by setting $\P_\K(y) := x$ where $\{x\} = \Proj_\K(y)$ for every
	$y \in \H$ such that $\d_\K(y) < r$.
\end{Def}

The following characterizations of prox-regularity are very useful. The proofs can be found in
\cite[Theorem 4.1]{PolRocThi00} and in \cite[Theorem 16]{ColThi10}.

\begin{Thm}\label{charact proxreg}
	Let $\K$ be a nonempty closed subset of $\H$ and let $r \in \opint{0,\infty}$. The following
	statements are equivalent.
	\begin{itemize}
		\item[(i)] $\K$ is $r$-prox-regular.
		\item[(ii)] $\d_\K$ is differentiable in $\{y \in \H\ :\ 0 < \d_\K(y) < r\}$.
		\item[(iii)] For every $x \in \K$ and $n \in N_\K(x)$ we have
			\[
				\duality{n}{z-x} \le \frac{\norm{n}{}}{2r}\norm{z-x}{}^2 \qquad \forall z \in \K.
			\]
	\end{itemize}
\end{Thm}


\subsection{Functions of bounded variation and differential measures}\label{differential measures}

Let $I$ be an interval of $\ar$. The set of $\H$-valued continuous functions defined on $I$ is denoted by
$\Czero(I;\H)$. If $f \in \Czero(I;\H)$ then $\norm{f}{\infty} := \sup\{\norm{f(t)}{}\ :\ t \in I\}$ denotes it supremum norm.

\begin{Def}
	Given an interval $I \subseteq \ar$, a function $f : I \function \H$, and a subinterval $J \subseteq I$, the \emph{variation of $f$ on $J$} is defined by
	\begin{equation}\notag
		\pV(f,J) :=
		\sup\left\{
		\sum_{j=1}^{m} \norm{(f(t_{j-1})- f(t_{j})}{}\ :\ m \in \en,\ t_{j} \in J\ \forall j,\ t_{0} < \cdots < t_{m}
		\right\}.
	\end{equation}
	If $\pV(f,I) < \infty$ we say that \emph{$f$ is of bounded variation on $I$} and we set
	\[
		\BV(I;\H) := \{f : I \function \H\ :\ \pV(f,I) < \infty\}.
	\]
\end{Def}

It is well known that every $f \in \BV(I;\H)$ admits one-sided limits $f(t-), f(t+)$ at every point $t \in I$, with the convention that $f(\inf I-) := f(\inf I)$ if $\inf I \in I$, and $f(\sup I+) := f(\sup I)$ if $\sup I \in I$.
The family of Borel sets in $I$ is denoted by $\borel(I)$. If
$\mu : \borel(I) \function \clint{0,\infty}$ is a measure,
then the space of $\H$-valued functions which are integrable with respect to $\mu$ will be denoted by
$\L^1(I, \mu; \H)$ or simply by $\L^1(\mu; \H)$. For the theory of integration of vector valued functions we refer, e.g., to \cite[Chapter VI]{Lan93}. When $\mu = \leb^1$, where $ \leb^1$ is the one dimensional Lebesgue measure, we write $\L^1(I; \H) := \L^1(I,\mu; \H)$.

We recall that a \emph{$\H$-valued measure on $I$} is a map $\nu : \borel(I) \function \H$ such that
$\nu(\bigcup_{n=1}^{\infty} B_{n})$ $=$ $\sum_{n = 1}^{\infty} \nu(B_{n})$ for every sequence $(B_{n})$ of mutually disjoint sets in $\borel(I)$. The \emph{total variation of $\nu$} is the positive measure
$\vartot{\nu} : \borel(I) \function \clint{0,\infty}$ defined by
\begin{align}\label{tot var measure}
	\vartot{\nu}(B)
	:= \sup\left\{\sum_{n = 1}^{\infty} \norm{\nu(B_{n})}{}\ :\
	B = \bigcup_{n=1}^{\infty} B_{n},\ B_{n} \in \borel(I),\
	B_{h} \cap B_{k} = \varnothing \text{ if } h \neq k\right\}. \notag
\end{align}
The vector measure $\nu$ is said to be \emph{with bounded variation} if $\vartot{\nu}(I) < \infty$. In this case the equality
$\norm{\nu}{} := \vartot{\nu}(I)$ defines a complete norm on the space of measures with bounded variation (see, e.g.
\cite[Chapter I, Section  3]{Din67}).

If $\mu : \borel(I) \function \clint{0,\infty}$ is a positive bounded Borel measure and if $g \in \L^1(I,\mu;\H)$, then
$g\mu : \borel(I) \function \H$ denotes the vector measure defined by
\begin{equation}\label{gmu}
	g\mu(B) := \int_B g\de \mu, \qquad B \in \borel(I). \notag
\end{equation}

Assume that $\nu : \borel(I) \function \H$ is a vector measure with bounded variation and $f : I \function \H$ and
$\phi : I \function \mathbb{R}$ are two \emph{step maps with respect to $\nu$}, i.e. there exist
$f_{1}, \ldots, f_{m} \in \H$, $\phi_{1}, \ldots, \phi_{m} \in \H$ and $A_{1}, \ldots, A_{m} \in \borel(I)$ mutually disjoint such that $\vartot{\nu}(A_{j}) < \infty$ for every $j$ and $f = \sum_{j=1}^{m} \indicator_{A_{j}} f_{j}$,
$\phi = \sum_{j=1}^{m} \indicator_{A_{j}} \phi_{j}.$ Here $\indicator_{S} $ is the characteristic function of a set $S$, i.e. $\indicator_{S}(x) := 1$ if $x \in S$ and $\indicator_{S}(x) := 0$ if $x \not\in S$. For such step maps we define
$\int_{I} \duality{f}{\de\nu} := \sum_{j=1}^{m} \duality{f_{j}}{\nu(A_{j})} \in \mathbb{R}$ and
$\int_{I} \phi \de \nu := \sum_{j=1}^{m} \phi_{j} \nu(A_{j}) \in \H$.

If $\Step(\vartot{\nu};\H)$ (resp. $\Step(\vartot{\nu})$) is the set of $\H$-valued (resp.\ real valued) step maps with respect to $\nu$, then the maps
$\Step(\vartot{\nu};\H)$ $\function$ $\H : f \longmapsto \int_{I} \duality{f}{\de\nu}$ and
$\Step(\vartot{\nu})$ $\function$ $\H : \phi \longmapsto \int_{I} \phi \de \nu$
are linear and continuous when $\Step(\vartot{\nu};\H)$ and $\Step(\vartot{\nu})$ are endowed with the
$\L^{1}$-seminorms $\norm{f}{\L^{1}(\vartot{\nu};\H)} := \int_I \norm{f}{} \de \vartot{\nu}$ and
$\norm{\phi}{\L^{1}(\vartot{\nu})} := \int_I |\phi| \de \vartot{\nu}$. Therefore they admit unique continuous extensions
$\mathsf{I}_{\nu} : \L^{1}(\vartot{\nu};\H) \function \mathbb{R}$ and
$\mathsf{J}_{\nu} : \L^{1}(\vartot{\nu}) \function \H$,
and we set
\[
	\int_{I} \duality{f}{\de \nu} := \mathsf{I}_{\nu}(f), \quad
	\int_{I} \phi\, \de\nu := \mathsf{J}_{\nu}(\phi),
	\qquad f \in \L^{1}(\vartot{\nu};\H),\quad \phi \in \L^{1}(\vartot{\nu}).
\]
Using step functions it is easy to check that
$|\int_{I} \duality{f}{\de \nu}| \le \int_{I} \norm{f(s)}{}\de \vartot{\nu}(s)$ for $f \in \L^{1}(\vartot{\nu};\H)$.
The following results (cf., e.g., \cite[Section III.17.2-3, p. 358-362]{Din67}) provide the connection between functions with bounded variation and vector measures which will be implicitly used in this paper.

\begin{Thm}\label{existence of Stietjes measure}
	For every $f \in \BV(I;\H)$ there exists a unique vector measure of bounded variation
	$\nu_{f} : \borel(I) \function \H$ such that
	\begin{align}
		\nu_{f}(\opint{c,d}) = f(d-) - f(c+), \qquad \nu_{f}(\clint{c,d}) = f(d+) - f(c-), \notag \\
		\nu_{f}(\clsxint{c,d}) = f(d-) - f(c-), \qquad \nu_{f}(\cldxint{c,d}) = f(d+) - f(c+). \notag
	\end{align}
	whenever $\inf I \le c < d \le \sup I$ and the left hand side of each equality makes sense.

	\noindent Conversely, if $\nu : \borel(I) \function \H$ is a vector measure with bounded variation, and if
	$f_{\nu} : I \function \H$ is defined by $f_{\nu}(t) := \nu(\clsxint{\inf I,t} \cap I)$, then $f_{\nu} \in \BV(I;\H)$ and
	$\nu_{f_{\nu}} = \nu$.
\end{Thm}

\begin{Prop}
	Let $f  \in \BV(I;\H)$, let $g : I \function \H$ be defined by $g(t) := f(t-)$, for $t \in \Int(I)$, and by $g(t) := f(t)$, if
	$t \in \partial I$, and let $V_{g} : I \function \mathbb{R}$ be defined by $V_{g}(t) := \pV(g, \clint{\inf I,t} \cap I)$. Then
	$\nu_{g} = \nu_{f}$ and $\vartot{\nu_{f}}(I) = \nu_{V_{g}}(I) = \pV(g,I)$.
\end{Prop}

The measure $\nu_{f}$ is called the \emph{Lebesgue-Stieltjes measure} or \emph{differential measure} of $f$. Let us see the connection between  the differential measure and the distributional derivative. If $f \in \BV(I;\H)$ and if
$\overline{f}: \mathbb{R} \function \H$ is defined by
\begin{equation}\label{extension to R}
	\overline{f}(t) :=
	\begin{cases}
		f(t)      & \text{if $t \in I$},                                              \\
		f(\inf I) & \text{if $\inf I \in \mathbb{R}$, $t \not\in I$, $t \le \inf I$}, \\
		f(\sup I) & \text{if $\sup I \in \mathbb{R}$, $t \not\in I$, $t \ge \sup I,$}
	\end{cases}
\end{equation}
then, as in the scalar case, it turns out (cf. \cite[Section 2]{Rec11a}) that $\nu_{f}(B) = \D \overline{f}(B)$ for every
$B \in \borel(\mathbb{R})$, where $\D\overline{f}$ is the distributional derivative of $\overline{f}$, i.e.
\[
	- \int_\mathbb{R} \varphi'(t) \overline{f}(t) \de t = \int_{\mathbb{R}} \varphi \de \D \overline{f},
	\qquad \forall \varphi \in \Czero_{c}^{1}(\mathbb{R};\mathbb{R}),
\]
where $\Czero_{c}^{1}(\mathbb{R};\mathbb{R})$ is the space of continuously differentiable functions on
$\mathbb{R}$ with compact support. Observe that $\D \overline{f}$ is concentrated on $I$:
$\D \overline{f}(B) = \nu_f(B \cap I)$ for every $B \in \borel(I)$, hence in the remainder of the paper, if
$f \in \BV(I,\H)$ then we will simply write
\begin{equation}
	\D f := \D\overline{f} = \nu_f,
\end{equation}
and from the previous discussion it follows that
\begin{equation}\label{D-TV-pV}
	\norm{\D f}{} = \vartot{\D f}(I) = \norm{\nu_f}{}
\end{equation}
and
\begin{equation}\label{D-TV-pV-2}
	\vartot{\D f}(I) =  \pV(f,I) \quad \text{if $f(t) \in \{(1-\lambda)f(t-) + \lambda f(t+)\ :\ \lambda \in \clint{0,1}\} \ \forall t \in I$}.
\end{equation}


%


\section{Main results}\label{Main res}

In this section we state the main results of the paper. Let us start by providing the rigorous formulation of the problem we are going to solve.

\begin{Pb}\label{Pb}
	Assume that \eqref{H-prel} holds. Given $T > 0$, $\C$, and $y_{0} \in \C(0)$, one has to find a function $y \in  \BV(\clint{0,T};\H)$ such that there exist a measure $\mu : \borel(\clint{0,T}) \function \clsxint{0,\infty}$ and a function $v \in \L^{1}(\mu;\H)$ satisfying
	\begin{alignat}{3}
		 & y(t) \in \C(t)                 & \qquad & \forall t \in \clint{0,T}, \label{sp-cstr}                 \\
		 & \!\D y = v \mu, \label{Dy = w}                                                                       \\
		 & v(t) \in -N_{\C(t)}(y(t))      & \qquad & \text{for $\mu$-a.e. $t \in \clint{0,T}$}, \label{BV-d.i.} \\
		 & y(0) = y_{0}. \label{sp-i.c.}
	\end{alignat}
\end{Pb}

In order to state our results we recall the precise notion of excess:
\begin{Def}
	If $\parti(\H)$ denotes the family of all subsets of $\H$, then the \emph{excess} is the function
	$\e : \parti(\H)\times\parti(\H) \function \clint{0,\infty}$ defined by
	\begin{equation}\notag
		\e(\Z_1,\Z_2) := \sup_{z_1 \in \Z_1} \d(z_1,\Z_2) =
		\inf\{\rho \ge 0\ :\ \Z_1 \subseteq \Z_2 + \overline{B}_\rho(0)\}, \qquad \Z_1, \Z_2 \in \parti(\H).
	\end{equation}
	The extended number $\e(\Z_1,\Z_2)$ is called \emph{excess of $\Z_1$ over $\Z_2$}.
\end{Def}
Observe that $\e(\void, \Z) = 0$ for every $\Z \in \parti(\H)$ and $\e(\Z, \void) = \infty$ for every
$\Z \in \parti(\H) \setmeno \{\void\}$. The excess $\e$ is an asymmetric semidistance, in particular the following properties hold (see \cite[Section 2a]{Mor74b}):
\begin{align*}
	 & \e(\Z_1,\Z_2) = 0 \ \Longleftrightarrow \Z_1 \subseteq \overline{\Z_2}                                    &
	 & \forall \Z_1, \Z_2 \in \parti(\H),                                                                          \\
	 & \e(\Z_1,\Z_2) \le \e(\Z_1,\Z_3) + \e(\Z_3,\Z_2)                                                           &
	 & \forall \Z_1, \Z_2, \Z_3 \in \parti(\H),                                                                    \\
	 & \e(\Z_1,\Z_2) = \e(\overline{\Z_1},\Z_2) = \e(\Z_1,\overline{\Z_2}) = \e(\overline{\Z_1},\overline{\Z_2}) &
	 & \forall \Z_1, \Z_2 \in \parti(\H).
\end{align*}

The main existence theorem is the following. It improves the existing result \cite[Theorem~4.2]{ColMon03} in two ways: first,
the continuity of the multivalued mapping $t \longmapsto \partial\C(t)$ is completely dropped; second,
the continuity of the driving set $t \longmapsto \C(t)$ is assumed with respect to the asymmetric distance $\e$, the excess, and not with respect to the Hausdorff distance, thereby emphasizing the unilateral character of the rating of the data, as observed by Moreau in its celebrated paper \cite{Mor77} in the convex framework. As a consequence we are able to provide an existence/uniqueness result for a much wider class of moving constraints $\C(t)$.

\begin{Thm}\label{main th1}
	Let $\C : \clint{0,T} \function \Conv_r(\H)$, a function with values in the family of $r$-prox-regular subsets of $\H$, satisfying the condition
	\begin{equation}\label{e-cont->e-unifcont}
		\forall \eps > 0\  \exists \delta > 0 \ :\ [0 \le t - s < \delta \ \Longrightarrow\  \e(\C(s),\C(t)) < \eps].
	\end{equation}
	If there exist $R > 0$ and $d > 0$ such that for every $t \in \clint{0,T}$
	\begin{equation}\label{unif. int. cone main thm}
		\forall x \in \partial\C(t) \ \exists z \in \C(t) \cap B_d(x) \ :\
		(1-\lambda)x+\lambda B_R(z) \subseteq \C(t), \ \forall \lambda \in \clint{0,1},
	\end{equation}
	then for every  $y_0 \in \C(0)$ there exists a unique solution $y \in \BV(\clint{0,T};\H)$ of Problem
	\ref{Pb}. Moreover $y \in \Czero(\clint{0,T};\H)$ and $\V(y,\clint{0,T}) \le C$
	for a constant $C > 0$ depending only on $r$, $R > 0$, $d > 0$, and on the continuity modulus of $\C$.
\end{Thm}

Condition \eqref{unif. int. cone main thm} can be referred to as a ``uniform inner cone condition'', which is a regularity assumption on the boundary of $\C(t)$. This regularity assumption on the boundary can be dropped if one assumes that there exists a fixed ball
contained on all the sets $\C(t)$, but this ball should be big enough (cf.\ the compatibility condition \eqref{compcond main thm}). This result in contained in the following theorem which generalizes \cite[Theorem~ 4.1]{ColMon03} by assuming the continuity of $\C(t)$ with respect to the excess only.

\begin{Thm}\label{main th2}
	Let $\C : \clint{0,T} \function \Conv_r(\H)$, a function with values in the family of $r$-prox-regular subsets of $\H$, satisfying the condition
	\begin{equation}\label{e-cont->e-unifcont.}
		\forall \eps > 0\  \exists \delta > 0 \ :\ [0 \le t - s < \delta \ \Longrightarrow\  \e(\C(s),\C(t)) < \eps].
	\end{equation}
	Assume that $y_0 \in \C(0)$.
	If
	\begin{equation}
		\exists w \in \H, \ \exists \rho > 0\ :\ \big[\,B_\rho(w) \subseteq \C(t)\  \forall t \in \clint{0,T}\,\big],
	\end{equation}
	and
	\begin{equation}\label{compcond main thm}
		(\norm{y_0 - w}{} + \rho)^2 < 2r\rho,
	\end{equation}
	then there exists a unique solution $y \in \BV(\clint{0,T};\H)$ of Problem \ref{Pb}. Moreover $y \in \Czero(\clint{0,T};\H)$ and
	$\V(y,\clint{0,T}) \le C$ for a constant $C > 0$ depending only on $r$, $w$, $\rho$, and on the continuity modulus of $\C$.
\end{Thm}

The proofs of Theorem \ref{main th1} and Theorem \ref{main th2} are contained in the following section \ref{proofs}.


\section{Proofs}\label{proofs}

\subsection{Catching-up algorithm}

In the following lemma we describe the Euler implicit discretization of \eqref{BV-d.i.}. This discretization scheme
(\eqref{implicit scheme for Lip-sweeping processes1}--\eqref{implicit scheme for Lip-sweeping processes3}) is usually called ``catching-up algorithm'' (cf. \cite[p. 353]{Mor77}), due to its evident geometrical meaning in the case of sweeping processes (see \eqref{stepproj} below).

\begin{Lem}\label{catchup}
	For every $t \in \clint{0,T}$ let $\C(t)$ be an $r$-prox-regular subset of $\H$, and assume that \eqref{e-cont->e-unifcont} holds. If $(\eps_n)_{n\in \en}$ is a sequence in $\ar$ such that
	\begin{equation}\label{epsn-1}
		0 < \eps_{n+1} < \eps_n < r \quad \forall n \in \en, \qquad  \lim_{n \to \infty} \eps_n = 0,
	\end{equation}
	then there exists a sequence $(\delta_n)_{n \in \en}$ in $\ar$, such that
	\begin{align}
		 & 0 < \delta_{n+1} < \delta_n \quad \forall n \in \en, \qquad \lim_{k \to \infty}\delta_k = 0, \label{deltan} \\
		 & \sup\{\e(\C(s),\C(t))\ :\ 0 \le t - s \le \delta_n\}  < \eps_n < r \quad \forall n \in \en. \label{CforP}
	\end{align}
	Moreover for every $n \in \en$ let $\mathfrak{S}_n$ be a subdivision of $\clint{0,T}$ satisfying the conditions
	\begin{align}
		 & \mathfrak{S}_n = \{t^n_j \in \clint{0,T}\ :\ j=0, \ldots, J^n\} \text{ for some $J_n \in \en$}, \label{subd1} \\
		 & t_0^n = 0,\  t_{J_n}^n = T,\quad  t^n_{j-1} < t_j^n \ \forall j \in \{0, \ldots, J^n\}, \quad
		\max_{j=1,\ldots,J^n} (t_j^n - t^n_{j-1}) \le \delta_n, \label{subd3}                                            \\
		 & \mathfrak{S}_n \subseteq \mathfrak{S}_{n+1}. \label{subd4}
	\end{align}
	Then for every $y_0 \in \C(0)$ and every $n \in \en$ there exists a family $(y_j^n)_{j =0}^{J^n}$ in $\H$ such that
	\begin{alignat}{3}
		 & y_0^n := y_0,
		\label{implicit scheme for Lip-sweeping processes1}                                                                       \\
		 & y_j^n \in \C(t_j^n)                                                   & \qquad & \forall j \in \en, \ 0 \le j \le J^n,
		\label{implicit scheme for Lip-sweeping processes2}                                                                       \\
		 & \dfrac{y_j^n -y_{j-1}^n}{t_j^n - t_{j-1}^n} \in -N_{\C(t_j^n)}(y_j^n) & \qquad & \forall j \in \en, \ 1 \le j \le J^n.
		\label{implicit scheme for Lip-sweeping processes3}
	\end{alignat}
	Under the above assumptions, condition \eqref{implicit scheme for Lip-sweeping processes3} is equivalent to
	\begin{equation}\label{stepproj}
		y_j^n = \P_{\C(t_j^n)}(y_{j-1}^n) \qquad  \forall j \in \en, \ 1 \le j \le J^n.
	\end{equation}
\end{Lem}

\begin{proof}
	The existence of the sequence $\delta_n$ satisfying \eqref{deltan} and \eqref{CforP} is a straightforward consequence of \eqref{e-cont->e-unifcont}. Now for every $n \in \en$ set
	\[
		y_0^n := y_0 \in \C(0) = \C(t_0^n),
	\]
	thus from \eqref{CforP} we have that $\d(y_0^n,\C(t_1^n)) \le \e(\C(a),\C(t^n_1)) < r$, and by the
	$r$-prox-regularity of $\C(t_1^n)$ we infer that there exists $y_1^n := \P_{\C(t_1^n)}(y_0^n) \in \C(t_1^n)$, so that $y_0^n - y_1^n \in N_{\C(t_1^n)}(t_1^n)$. For any integer $j$ such that $2 \le j \le J_n$, let us assume that $y_{j-1}^n$ satisfies the inclusions $y_{j-1}^n \in \C(t_{j-1}^n)$ and
	$y_{j-2}^n-y_{j-1}^n \in N_{\C(t_{j-1}^n)}(y_{j-1}^n)$. Hence by virtue of \eqref{CforP} we have
	$
		\d(y_{j-1}^n, \C(t_j^n)) \le \e(\C(t_{j-1}^n),\C(t_j^n)) < r,
	$
	so that there exists
	\begin{equation}\label{LT3}
		y_j^n := \P_{\C(t_j^n)}(y_{j-1}^n) \qquad \forall j \in \en,
	\end{equation}
	and $y_{j-1}^n-y_j^n  \in N_{\C(t_j^n)}(y_j^n)$. Therefore, recalling that $N_{\C(t)}(z)$ is a cone for every
	$z \in \C(t)$ and for every $t \in \clint{0,T}$, we have recursively defined a family $(y_j^n)_{j =0}^{J^n}$ in $\H$ satisfying
	\eqref{implicit scheme for Lip-sweeping processes1}--\eqref{implicit scheme for Lip-sweeping processes3}.
\end{proof}

\begin{Lem}\label{exact discr.sol}
	Under the assumptions of Lemma \ref{catchup}, let $(\eps_n)_{n\in \en}$ and $(\delta_n)_{n \in \en}$ be
	sequences in $\ar$ such that \eqref{epsn-1} and \eqref{deltan}--\eqref{CforP} hold, let
	$(\mathfrak{S}_n)_{n \in \en}$ be a sequence of subdivisions of $\clint{0,T}$ satisfying
	\eqref{subd1}--\eqref{subd4}, and let $(y_j^n)_{j=1}^{J^n}$ be the family in $\H$ satisfying
	\eqref{implicit scheme for Lip-sweeping processes1}--\eqref{implicit scheme for Lip-sweeping processes3}.
	For every $n \in \en$ let $y_n : \clint{0,T} \function \H$ be
	the right continuous step function defined by
	\begin{align}
		 & y_n(0) := y_0^n = y_0, \label{yn(0)}                                                                                          \\
		 & y_n(t) := y_{j-1}^n \qquad \text{if $t \in \clsxint{t_{j-1}^{n}, t_{j}^{n}}$}, \quad j \in \en, \ 1 \le j \le J^n, \label{yn}
	\end{align}
	which we will call \emph{discrete solution of Problem \ref{Pb}},
	and let $x_n : \clint{0,T} \function \H$ be the piecewise affine interpolant of $(y_j^n)_{j=0}^{J^n}$, i.e.
	\begin{equation}\label{x_n}
		x_n(t) := y_{j-1}^n + \frac{t-t^n_{j-1}}{t_j^n - t_{j-1}^n}(y_j^n - y_{j-1}^n)
		\qquad \text{if $t \in \clsxint{t_{j-1}^n, t_j^n}$}, \quad j \in \en, \ 1 \le j \le J^n.
	\end{equation}
	Then
	\begin{equation}\label{xn-yn to 0}
		\lim_{n \to \infty} \norm{y_n - x_n}{\infty} = 0,
	\end{equation}
	and $y_n$ solves Problem \ref{Pb} with $\C(t)$ replaced by $\C^n(t)$ defined by
	\[
		\C^n(t) := \C(t_{j-1}^n) \qquad \text{if $t \in \clsxint{t_{j-1}^{n}, t_{j}^{n}}$},
	\]
	i.e. if $\mu_n : \borel(\clint{0,T}) \function \clsxint{0,\infty}$ and $v_n : \clint{0,T} \function \H$ are defined by
	\[
		\mu_n := \sum_{j=1}^{J^n}\delta_{t_j^n}, \qquad
		v_n(t) := \sum_{j=1}^{J^n}(y_j^n - y_{j-1}^n)\indicator_{\{t_j^n\}}(t) = \D y_n(t), \qquad t \in \clint{0,T},
	\]
	then $v_n \in \L^1(\mu_n;\H)$ and
	\[
		\D y_n(t) = v_n(t) \mu_n, \quad
		v_n(t) \in -N_{\C^n(t)}(y_n(t)) \qquad \text{$\forall t \in \clint{0,T}$}.
	\]
\end{Lem}

\begin{proof}
	Let us observe that by \eqref{implicit scheme for Lip-sweeping processes2}, \eqref{subd3}, and
	\eqref{CforP}, we have that
	\[
		\norm{y^n_j - y^n_{j-1}}{} = \d_{\C(t_j^{n})}(y_{j-1}^{n}) \le \e(\C(t_{j-1}^{n}),\C(t_j^{n})) < \eps_n
		\qquad \forall j \in \en, \ 1 \le j \le J^n,
	\]
	thus \eqref{xn-yn to 0} follows from the equality
	\[
		\norm{y_n(t) - x_n(t)}{} =
		\bigg|\frac{t - t_{j-1}^n}{ t_j^n - t_{j-1}^n }\bigg|
		\norm{y_j^n - y_{j-1}^n}{} \qquad
		\forall t \in \cldxint{t_{j-1}^n, t_j^n}.
	\]
	If $0 \le j \le J^n$, from \eqref{implicit scheme for Lip-sweeping processes3} it follows that
	\[
		\D y_n(\{t_j^n\}) = y_j^n - y_{j-1}^n \in -N_{\C(t_j^n)}(y_j^n) = N_{\C(t_j^n)}(y_n(t_j^n)),
	\]
	and we also trivially have that
	\[
		\D y_n(\opint{t_{j-1}^n,t_j^n}) = 0 \in N_{\C(t)}(y_n(s)) \qquad
		\forall s \in \opint{t_{j-1}^n,t_j^n}. \qedhere
	\]
\end{proof}


\subsection{Uniform convergence of discrete approximations}

In the following theorem we discuss the uniform convergence of the approximated solutions.
It is worth noting that one of the main innovations is the way we obtain the Cauchy estimate,
due to the fact that the $\C$ is assumed to be continuous with respect to the excess, and not with respect to the Hausdorff distance.

\begin{Thm}\label{T:unif conv}
	For every $t \in \clint{0,T}$ let $\C(t)$ be an $r$-prox-regular subset of $\H$, and assume that \eqref{e-cont->e-unifcont} holds. Fix $y_0 \in \C(0)$.
	Let  $(\eps_n)_{n\in \en}$ be a sequence in $\ar$ such that
	\begin{equation}\label{epsn-2}
		0 < \eps_{n+1} < \eps_n < r \quad \forall n \in \en, \qquad  \sum_{n = 0}^\infty \eps_n < \infty.
	\end{equation}
	Let $(\delta_n)_{n \in \en}$ be a sequence in $\ar$ satisfy \eqref{deltan}--\eqref{CforP}, let
	$(\mathfrak{S}_n)_{n \in \en}$ be a sequence of subdivisions of $\clint{0,T}$ satisfying
	\eqref{subd1}--\eqref{subd4}, and let $(y_j^n)_j^{J^n}$ be the family in $\H$ satisfying
	\eqref{implicit scheme for Lip-sweeping processes1}--\eqref{implicit scheme for Lip-sweeping processes3}, all of them existing by virtue of  Lemma \ref{catchup}. For every $n \in \en$, let $y_n : \clint{0,T} \function \H$
	be the discrete solution of Problem \ref{Pb}, i.e.\ the right continuous step function defined by \eqref{yn(0)}--\eqref{yn}
	and let $x_n : \clint{0,T} \function \H$ be the piecewise affine interpolant of $(y_j^n)_{j=0}^{J^n}$ defined by
	\eqref{x_n}.
	If
	\[
		\sup_{n \in \en}\V(y_n,\clint{0,T}) < \infty,
	\]
	then $(y_n)_{n \in \en}$ and $(x_n)_{n \in \en}$ are uniformly convergent to a function
	$y \in \BV(\clint{a,b};\H) \cap \Czero(\clint{0,T};\H)$ such that $y(0) = y_0$ and $y(t) \in \C(t)$ for every $t \in \clint{0,T}$.
\end{Thm}

\begin{proof}
	We will assume without loss of generality that $\eps_n \le 1$ for all $n \in \en$.
	Let us recall that observe that by \eqref{implicit scheme for Lip-sweeping processes2}, \eqref{subd3}, and
	\eqref{CforP}, we have that
	\begin{equation}\label{yj - yj-1}
		\norm{y^n_j - y^n_{j-1}}{} = \d_{\C(t_j^{n})}(y_{j-1}^{n}) \le \e(\C(t_{j-1}^{n}),\C(t_j^{n})) < \eps_n
		\qquad \forall j \in \en, \ 1 \le j \le J^n.
	\end{equation}
	Let us also define the \emph{time anticipating} function $\theta_n : \clint{0,T} \function \clint{0,T}$ by
	\begin{equation}\label{tetan}
		\theta_n(0) := 0, \qquad \theta_n(t) := t_{j}^n \quad \text{if $t \in \cldxint{t_{j-1}^n,t_j^n}$}, \ j \in \en,\ 1 \le j \le J^n,
	\end{equation}
	so that
	\begin{equation}\label{LT4}
		0 \le \theta_n(t) - t \le \delta_n \qquad \forall t \in \clint{0,T},
	\end{equation}
	and
	\begin{equation}
		x_n(\theta_n(t)) \in \C(\theta_n(t)) \qquad \forall t \in \clint{0,T}.
	\end{equation}
	By \eqref{x_n},\eqref{tetan} and \eqref{implicit scheme for Lip-sweeping processes2}, if $t \in \cldxint{t_{j-1}^n,t_j^n}$ we have $\norm{x_{n}(\theta_{n}(t)) - x_{n}(t)}{} \le \norm{y^n_j - y^n_{j-1}}{}$, thus, thanks to \eqref{yj - yj-1},
	\begin{equation}\label{LT-xn-x}
		\norm{x_{n}(\theta_{n}(t)) - x_{n}(t)}{} \le \eps_n \qquad \forall t \in \clint{0,T}.
	\end{equation}
	We have, by \eqref{implicit scheme for Lip-sweeping processes3}, that
	\begin{equation}\label{x_n' in N}
		x_n'(t) = -\frac{y_{j-1}^n - y_j^n}{t_j^n - t_{j-1}^n} \in N_{C(t_j^n)}(y_j^n) = N_{C(\theta_n(t))}(x(\theta_n(t)))
		\qquad \forall t \in \opint{t_{j-1}^n, t_j^n}, \ \forall j \in \en.
	\end{equation}
	Now let us fix $n \in \en$ and $t \in \clint{0,T}$ such that $t \neq t_j^n$ and $t \neq t_{j}^{n+1}$ for all $j \in \{0,\ldots,J^n\}$. We have that
	\begin{align}
		2\frac{\de }{\de t}\norm{x_{n+1}(t) - x_n(t)}{}^2
		 & = \lduality{x'_{n+1}(t) - x'_n(t)}{x_{n+1}(t) - x_n(t)} \notag                                              \\
		 & = \lduality{x'_{n+1}(t) - x'_n(t)}{x_{n+1}(t) - x_{n+1}(\theta_{n+1}(t))} \notag                            \\
		 & \phantom{=\ }  +  \lduality{x'_{n+1}(t) - x'_n(t)}{x_{n+1}(\theta_{n+1}(t)) - x_{n}(\theta_{n}(t))} \notag  \\
		 & \phantom{=\ }  +  \lduality{x'_{n+1}(t) - x'_n(t)}{x_{n}(\theta_{n}(t)) - x_n(t)} \notag                    \\
		 & \le \norm{x'_{n+1}(t) - x'_n(t)}{}\norm{x_{n+1}(t) - x_{n+1}(\theta_{n+1}(t))}{} \notag                     \\
		 & \phantom{=\ }  +  \lduality{x'_{n+1}(t) - x'_n(t)}{x_{n+1}(\theta_{n+1}(t)) - x_{n}(\theta_{n}(t))} \notag  \\
		 & \phantom{=\ }  +  \norm{x'_{n+1}(t) - x'_n(t)}{}\norm{x_{n}(\theta_{n}(t)) - x_n(t)}{} \notag               \\
		 & \le (\norm{x'_{n+1}(t)}{} + \norm{x'_n(t)}{})(\eps_{n+1}+\eps_{n}) \notag                                   \\
		 & \phantom{=\ }  +  \lduality{x'_{n+1}(t) - x'_n(t)}{x_{n+1}(\theta_{n+1}(t)) - x_{n}(\theta_{n}(t))}, \notag
	\end{align}
	therefore
	\begin{align}
		 & 2\norm{x_{n+1}(t) - x_n(t)}{}^2 \notag                                                               \\
		 & =  \int_0^t  \lduality{x'_{n+1}(s) - x'_n(s)}{x_{n+1}(s) - x_n(s)} \de s \notag                      \\
		 & \le (\eps_{n+1}+\eps_{n})\int_0^t(\norm{x'_{n+1}(s)}{} + \norm{x'_n(s)}{}) \de s \notag              \\
		 & \phantom{\le\ } +
		\int_0^t \lduality{x'_{n+1}(s) - x'_n(s)}{x_{n+1}(\theta_{n+1}(s)) - x_{n}(\theta_{n}(s))} \de s \notag \\
		 & \le 4V\eps_{n} +
		\int_0^t \lduality{x'_{n+1}(s) - x'_n(s)}{x_{n+1}(\theta_{n+1}(s)) - x_{n}(\theta_{n}(s))} \de s.\label{d/dt <}
	\end{align}
	Before proceeding with the estimate of the right-hand side of \eqref{d/dt <}, let us preliminarily observe that  from \eqref{LT-xn-x} we obtain
	\begin{align}
		 & \norm{x_{n+1}(\theta_{n+1}(t)) - x_{n}(\theta_{n}(t))}{}^2 \notag \\
		 & \le  3\norm{x_{n+1}(\theta_{n+1}(t)) - x_{n+1}(t)}{}^2 +
		3\norm{x_{n+1}(t) - x_{n}(t)}{}^2 +
		3\norm{x_{n}(t) - x_{n}(\theta_{n}(t))}{}^2 \notag                   \\
		 & \le  3\eps_{n+1}^2 + 3\norm{x_{n+1}(t) - x_{n}(t)}{}^2 +
		3\eps_{n}^2 \le
		6\eps_{n}^2  + 3\norm{x_{n+1}(t) - x_{n}(t)}{}^2. \label{stima freq}
	\end{align}
	Now let $j, k \in \en$ be such that $j \in \{1,\ldots,J^n\}$, $k \in \{1,\ldots,J^{n+1}\}$, and
	\begin{equation}\label{tjtk}
		t_{j-1}^n \le t_{k-1}^{n+1} < t < t_k^{n+1} \le t_j^n,
	\end{equation}
	thus $\theta_{n+1}(t) = t_k^{n+1}$ and $x_n(t_{j-1}^n) = y_{j-1}^{n} \in \C(t_{j-1}^{n})$, so that thanks to \eqref{subd3} and \eqref{CforP},
	\begin{equation}
		\d(x_n(t_{j-1}^n),\C(\theta_{n+1}(t))) \le \e(\C(t_{j-1}^{n}), \C(t_k^{n+1})) \le \eps_n < r. \label{LT6}
	\end{equation}
	Therefore there exists $\P_{\C(\theta_{n+1}(t))}(x_n(t_{j-1}^n))$ and, as by \eqref{x_n' in N} we have that $-x'_{n+1}(t)$ $\in$ $N_{\C(\theta_{n+1}(t))}(x_{n+1}(\theta_{n+1}(t)))$, using
	the $r$-prox-regularity of $\C(\theta_{n+1}(t))$, we can write
	\begin{align}
		 & \lduality{x'_{n+1}(t)}{x_{n+1}(\theta_{n+1}(t)) - x_{n}(\theta_{n}(t))} \notag                                                \\
		 & =  \lduality{x'_{n+1}(t)}{x_{n+1}(\theta_{n+1}(t)) - \P_{\C(\theta_{n+1}(t))}(x_n(t_{j-1}^n))}  \notag                        \\
		 & \phantom{=\ } +
		\lduality{x'_{n+1}(t)}{\P_{\C(\theta_{n+1}(t))}(x_n(t_{j-1}^n)) - x_{n}(\theta_{n}(t))} \notag                                   \\
		 & \le \frac{\lnorm{x_{n+1}(\theta_{n+1}(t)) - \P_{\C(\theta_{n+1}(t))}(x_n(t_{j-1}^n))}{}^2}{2r} \notag                         \\
		 & \phantom{\le\ } + \norm{x'_{n+1}(t)}{}\lnorm{\P_{\C(\theta_{n+1}(t))}(x_n(t_{j-1}^n))) - x_{n}(\theta_{n}(t))}{}. \label{LT5}
	\end{align}
	In the following two computations we estimate the right-hand side of \eqref{LT5}.
	First we observe that from \eqref{x_n},\eqref{tetan}, \eqref{yj - yj-1}, \eqref{stima freq}, and \eqref{LT6}, we infer that
	\begin{align}
		 & \lnorm{x_{n+1}(\theta_{n+1}(t)) - \P_{\C(\theta_{n+1}(t))}(x_n(t_{j-1}^n))}{}^2 \notag \\
		 & \le 3\lnorm{x_{n+1}(\theta_{n+1}(t)) - x_n(\theta_n(t))}{}^2 \notag                    \\
		 & \phantom{\le \ } +  3\lnorm{x_n(\theta_n(t)) - x_n(t_{j-1}^n)}{}^2  +
		3\lnorm{x_n(t_{j-1}^n) - \P_{\C(\theta_{n+1}(t))}(x_n(t_{j-1}^n))}{}^2 \notag             \\
		 & = 3\lnorm{x_{n+1}(\theta_{n+1}(t)) - x_n(\theta_n(t))}{}^2 \notag                      \\
		 & \phantom{\le \ } +  3\lnorm{y_j^n - y_{j-1}^n}{}^2  +
		3\lnorm{x_n(t_{j-1}^n) - \P_{\C(\theta_{n+1}(t))}(x_n(t_{j-1}^n))}{}^2 \notag             \\
		 & \le 3\lnorm{x_{n+1}(\theta_{n+1}(t)) - x_n(\theta_n(t))}{}^2  + 3\eps_n^2  +
		3\d(x_n(t_{j-1}^n), \C(\theta_{n+1}(t)))^2\notag                                          \\
		 & \le 18\eps_n^2 + 9\norm{x_{n+1}(t) - x_{n}(t)}{}^2 +
		3\eps_n^2  +3\eps_n^2, \label{LT7}
	\end{align}
	and that
	\begin{align}
		 & \lnorm{\P_{\C(\theta_{n+1}(t))}(x_n(t_{j-1}^n))) - x_{n}(\theta_{n}(t))}{} \notag                 \\
		 & \le \lnorm{\P_{\C(\theta_{n+1}(t))}(x_n(t_{j-1}^n))) - x_n(t_{j-1}^n)}{} +
		\lnorm{x_n(t_{j-1}^n) - x_{n}(\theta_{n}(t))}{}  \notag                                              \\
		 & = \d(x_n(t_{j-1}^n), \C(\theta_{n+1}(t))) +                   \lnorm{y_{j-1}^n - y_j^n}{}  \notag \\
		 & \le   \eps_n +
		\eps_n  = 2\eps_n, \label{LT8}
	\end{align}
	therefore, from \eqref{LT5}, \eqref{LT7}, and \eqref{LT8} we obtain that there exists a constant $K_1 > 0$, independent of $n$, such that
	\begin{align}
		 & \int_0^t\lduality{x'_{n+1}(t)}{x_{n+1}(\theta_{n+1}(t)) - x_{n}(\theta_{n}(t))}\de s \notag \\
		 & \le  \frac{9}{2r}\int_0^t\lnorm{x_{n+1}(s) - x_n(s)}{}^2 \de s  +
		\int_0^tK_1(1+\norm{x'_{n+1}(s)}{})\eps_n \de s \notag                                         \\
		 & \le  \frac{9}{2r}\int_0^t\lnorm{x_{n+1}(s) - x_n(s)}{}^2 \de s  +
		K_1(T+V)\eps_n. \label{LT9}
	\end{align}
	Now, recalling \eqref{tjtk}, observe that $\theta_n(t) = t_j^n$ and
	$x_{n+1}(\theta_{n+1}(t)) = y_{k-1}^{n+1} \in \C(t_{k-1}^{n+1})$, thus, thanks to \eqref{subd3} and \eqref{CforP},
	\begin{equation}
		\d(x_{n+1}(\theta_{n+1}(t)),\C(\theta_{n}(t)))
		\le \e(\C(t_{k-1}^{n+1}), \C(t_j^n)) < \eps_n. \label{LT10}
	\end{equation}
	Therefore there exists $\P_{\C(\theta(t))}(x_{n+1}(\theta_{n+1}(t)))$ and, as
	$-x'_{n}(t) \in N_{\C(\theta_{n}(t))}(x_{n}(\theta_{n}(t)))$ by  \eqref{x_n' in N}, using the $r$-prox-regularity of $\C(\theta_{n}(t))$, we can write
	\begin{align}
		 & \lduality{- x'_n(t)}{x_{n+1}(\theta_{n+1}(t)) - x_{n}(\theta_{n}(t))} \notag                                   \\
		 & =  \lduality{- x'_n(t)}{x_{n+1}(\theta_{n+1}(t)) - \P_{\C(\theta_n(t))}(x_{n+1}(\theta_{n+1}(t)))} \notag      \\
		 & \phantom{=\ } +
		\lduality{- x'_n(t)}{\P_{\C(\theta_n(t))}(x_{n+1}(\theta_{n+1}(t))) - x_{n}(\theta_{n}(t))} \notag                \\
		 & \le \norm{x'_n(t)}{}\lnorm{x_{n+1}(\theta_{n+1}(t)) - \P_{\C(\theta_n(t))}(x_{n+1}(\theta_{n+1}(t)))}{} \notag \\
		 & \phantom{\le\ } +
		\frac{\lnorm{\P_{\C(\theta_n(t))}(x_{n+1}(\theta_{n+1}(t))) - x_{n}(\theta_{n}(t))}{}^2}{2r}. \label{LT11}
	\end{align}
	In order to estimate the right-hand side of \eqref{LT11} let us observe that, by \eqref{LT10}, we have
	\begin{equation}
		\lnorm{x_{n+1}(\theta_{n+1}(t)) - \P_{\C(\theta_n(t))}(x_{n+1}(\theta_{n+1}(t)))}{}
		= \d(x_{n+1}(\theta_{n+1}(t))), \C(\theta_{n}(t)))
		\le \eps_n, \label{LT12}
	\end{equation}
	and that, using also \eqref{stima freq},
	\begin{align}
		 & \lnorm{\P_{\C(\theta_n(t))}(x_{n+1}(\theta_{n+1}(t))) - x_{n}(\theta_{n}(t))}{}^2 \notag   \\
		 & \le 2\lnorm{\P_{\C(\theta_n(t))}(x_{n+1}(\theta_{n+1}(t))) - x_{n+1}(\theta_{n+1}(t))}{}^2
		+  2\lnorm{x_{n+1}(\theta_{n+1}(t)) - x_n(\theta_n(t))}{}^2
		\notag                                                                                        \\
		 & \le 2\eps_n + 12\eps_n^2 + 6\norm{x_{n+1}(t) - x_{n}(t)}{}^2,
		\label{LT13}
	\end{align}
	therefore, from \eqref{LT11}, \eqref{LT12}, and \eqref{LT13} we obtain that there exists a constant $K_2 > 0$, independent of $n$, such that
	\begin{align}
		 & \int_0^t \lduality{-x'_{n}(s)}{x_{n+1}(\theta_{n+1}(s)) - x_{n}(\theta_{n}(s))} \de s \notag \\
		 & \le \frac{3}{r}\int_0^t \lnorm{x_{n+1}(s) - x_n(s)}{}^2 \de s  +
		\int_0^t K_2(1 + \norm{x_n'(s)}{})\eps_n \de s \notag                                           \\
		 & \le \frac{3}{r}\int_0^t \lnorm{x_{n+1}(s) - x_n(s)}{}^2 \de s  +
		K_2(T + V)\eps_n. \label{LT14}
	\end{align}
	Collecting together \eqref{d/dt <}, \eqref{LT9} and \eqref{LT14} we infer that there exists a constant $K_3 > 0$, independent of $n$ (depending only on $T$ and $V$), such that
	\begin{equation}
		\norm{x_{n+1}(t) - x_n(t)}{}^2
		\le K_3\eps_n + \frac{9}{2r}\norm{x_{n+1}(t) - x_{n}(t)}{}^2. \label{LT15}
	\end{equation}
	By \eqref{LT15} and by the Grönwall lemma we infer that there is a constant $K(T,V,r) > 0$ independent of $n$ (depending only on $r$, $T$ and $V$), such that
	\[
		\norm{x_{n+1}(t) - x_n(t)}{}^2 \le K(T,V,r)\eps_n \qquad \forall t \in \clint{0,T}, \quad \forall n \in \en,
	\]
	so that
	\[
		\sum_{n=1}^\infty \sup_{t \in \clint{0,T}}\norm{x_{n+1}(t) - x_n(t)}{} < \infty
	\]
	and the sequence $x_n$ is uniformly Cauchy in $\Czero(\clint{0,T};\H)$. It follows that there exists a
	continuous function $y : \clsxint{0,\infty} \function \H$ such that
	\begin{equation}\label{xn->y}
		x_n \to y \quad \text{uniformly on $\clint{0,T}, \qquad \forall T > 0$},
	\end{equation}
	thus $y(0) = y_0$. From \eqref{xn->y} and from \eqref{xn-yn to 0} of Lemma \ref{exact discr.sol} we also deduce that
	\begin{equation}\label{yn>x}
		y_n \to y \quad \text{uniformly on $\clint{0,T}$}.
	\end{equation}
	The assumption that $\sup_{n \in \en}\V(y_n,\clint{0,T}) < \infty$ and the lower semicontinuity of the variation w.r.t.\ the pointwise convergence yield $y \in \BV(\clint{0,T};\H)$.
	Finally if $t \in \clint{0,T}$ and $t \in \clsxint{t_{j-1}^n, t_j^n}$ for some $j$, then we have
	\begin{align}
		\d(y_n(t),\C(t)) & = \d(y_{j-1}^n, \C(t))  \le  \e(\C(t_{j-1}^n),\C(t))  \le \eps_n,
	\end{align}
	therefore by \eqref{yn>x} and by the closedness of $\C(t)$ we infer that $y(t) \in \C(t)$ for every $t \in \clint{0,T}$.
\end{proof}


\subsection{Uniform estimate of the variation}

This subsection is devoted to the uniform estimate of the variation of the approximated solutions.
We start by recalling the following lemma (cf. \cite[Lemma 2.2]{ColMon03}).

\begin{Lem}\label{L:TV-est}
	Assume that $m \in \en$, $r > 0$, and $\C_1, \ldots, \C_m$ are $r$-prox-regular subsets of $\H$ satisfying the inequalities
	\begin{equation}\label{hp1 TV-est}
		\e(\C_{k},\C_{k+1}) < r/2 \quad \forall k \in \{1,\ldots,m-1\},
	\end{equation}
	Moreover assume that there exist $w \in \H$ and $\rho > 0$ such that $B_\rho(w) \subseteq \C_k$ for all $k$.
	Let $y_0 \in \H$ be such that $\d_{\C_1}(y_0) < r/2$, and define
	\[
		y_k := \P_{\C_k}(y_{k-1}) \qquad k = 1,\ldots, m,
	\]
	\begin{equation}
		\alpha := \max\left\{\d_{\C_1}(y_0),\e(\C_1,\C_2),\ldots,\e(\C_{m-1 },\C_m)\right\} + \norm{y_0-w}{} + \rho.
	\end{equation}
	If
	\begin{equation}
		\alpha^2 < 2r\rho,
	\end{equation}
	then
	\begin{equation}\label{f-TV-est}
		\sum_{k=1}^m \norm{y_{k} - y_{k-1}}{} \le \max\left\{r\frac{\norm{y_0 - w}{}^2 - \rho^2}{2r\rho - \alpha^2}, 0\right\}.
	\end{equation}
\end{Lem}

Then, following exactly the argument in \cite[p. 57]{ColMon03}, we can estimate the variation of the approximated solutions in the case when the moving set $\C(t)$ contains a fixed ball at every time $t$.

\begin{Lem}\label{TVest-ball}
	For every $t \in \clint{0,T}$ let $\C(t)$ be an $r$-prox-regular subset of $\H$, and assume that \eqref{e-cont->e-unifcont} holds. Moreover assume that
	\begin{equation}
		\exists w \in \H, \ \exists \rho > 0\ :\ \big[\,B_\rho(w) \subseteq \C(t)\  \forall t \in \clint{0,T}\,\big],
	\end{equation}
	and
	\begin{equation}\label{compcond}
		(\norm{y_0 - w}{} + \rho)^2 < 2r\rho.
	\end{equation}
	Let $(\eps_n)_{n \in \en}$ and $(\delta_n)_{n\in \en}$ be sequences in $\ar$, and $(\mathfrak{S}_n)_{n \in \en}$ be a sequence of subdivisions of $\clint{0,T}$, satisfying \eqref{epsn-1}, \eqref{deltan}--\eqref{CforP}, and \eqref{subd1}--\eqref{subd4}, and let $(y_j^n)_{j =0}^{J^n}$ is the family in $\H$ defined in
	\eqref{implicit scheme for Lip-sweeping processes1}--\eqref{implicit scheme for Lip-sweeping processes3} by virtue of Lemma \ref{catchup}. If for every $n \in \en$, we let $y_n : \clint{0,T} \function \H$
	be the discrete solution of Problem \ref{Pb}, i.e.\ the right continuous step function defined by \eqref{yn(0)}--\eqref{yn},
	then
	\begin{equation}
		\sup_{n \in \en}\V(y_n,\clint{0,T}) < \infty,
	\end{equation}
	and the supremum in the previous formula is depending on $r$, $y_0$, $w$, $\rho$, and on the continuity  modulus of $\C$.
\end{Lem}

\begin{proof}
	Thanks to \eqref{compcond} we can find $n_r \in \en$ such that
	\begin{equation}
		(\norm{y_0-w}{} + \rho + \eps_n)^2 < 2r\rho \qquad \forall n \in \en,\ n \ge n_r.
	\end{equation}
	\begin{equation}
		\eps_n < r/2 \qquad \forall n \in \en,\ n \ge n_r.
	\end{equation}
	Therefore we conclude by observing that $\V(y_n,\clint{0,T}) = \sum_{j=1}^{J^n}\norm{y_j^n - y_{j-1}^n}{}$ and by applying Lemma \ref{L:TV-est} with $m = J^n$, $\C_j = \C(t_j^n)$ for $1 \le j \le J^n$,
	$\alpha = \norm{y_0 -w}{} + \rho + \sup_{n \ge n_r}\eps_n$,
\end{proof}

The next lemma is a variant of \cite[Lemma 2.3]{ColMon03} which will allows us to get rid of the assumption (3) of \cite[Theorem 4.2]{ColMon03}: namely, thanks to next lemma, in our main theorems we will not need any continuity in time of $\partial\C$, neither w.r.t.\ the Hausdorff distance nor w.r.t.\ the excess.

\begin{Lem}\label{Lunifball}
	Let $T>0$ and $t_0\in[0,T)$.
				For every $t \in \clint{0,T}$ let $\C(t)$ be an $r$-prox-regular subset of $\H$ such that
				\begin{enumerate}
					\item \label{it:prox-reg} $\C(t)$ is $r$-prox-regular for every $t\in[0,T]$;
					\item  \label{it:exc} $\e\bigl(\C(t_0),\C(t)\bigr) \leq \omega(t-t_0)$ for every $t \in \clint{t_0,T}$,
					      with $\omega$ non-decreasing and such that $\lim\limits_{t\to0^+} \omega(t) = 0$;
					\item  \label{it:ball} there exist $x\in \H$ and $\rho_0>0$ such that $B_{\rho_0}(x)\subseteq \C(t_0)$.
				\end{enumerate}
				Then for every $\rho\in(0,\rho_0)$ there exists $\tau=\tau(\omega,r,\rho_0,\rho)\in(0,T]$, depending only on $\omega,r,\rho_0,\rho$ and independent of $x$ and $t_0$, such that $B_\rho(x)\subseteq \C(t)$ for every $t \in \clint{t_0,\min\{t_0+\tau,T\}}$.
\end{Lem}

\begin{proof}
	Given $\rho\in(0,\rho_0)$, define $\eta:=\eta(r,\rho_0,\rho):=\min\{\rho_0-\rho,r\}$ and let
	$\tau=\tau(\omega,r,\rho_0,\rho)>0$ be such that $\omega(\tau)<\min\{\eta,\rho\}$.
	Observe that $\tau$ is independent of $x$ and $t_0$.

	We start by proving that $\d\bigl(x,\partial \C(t)\bigr) \geq \rho$ for every $t\in[t_0,\min\{t_0+\tau,T\}]$.
	If it were not the case, we would have some $t\in[t_0,\min\{t_0+\tau,T\}]$ and $y\in\partial \C(t)$ with $\norm{y-z}{}<\rho$.
	Since $\C(t)$ is $r$-prox-regular, and $y$ belongs to its boundary,
	we can find $z\in \H \setmeno \C(t)$ such that $B_\eta(z)\subseteq \H \setmeno \C(t)$ and $\norm{z-y}{} = \eta < r$.
	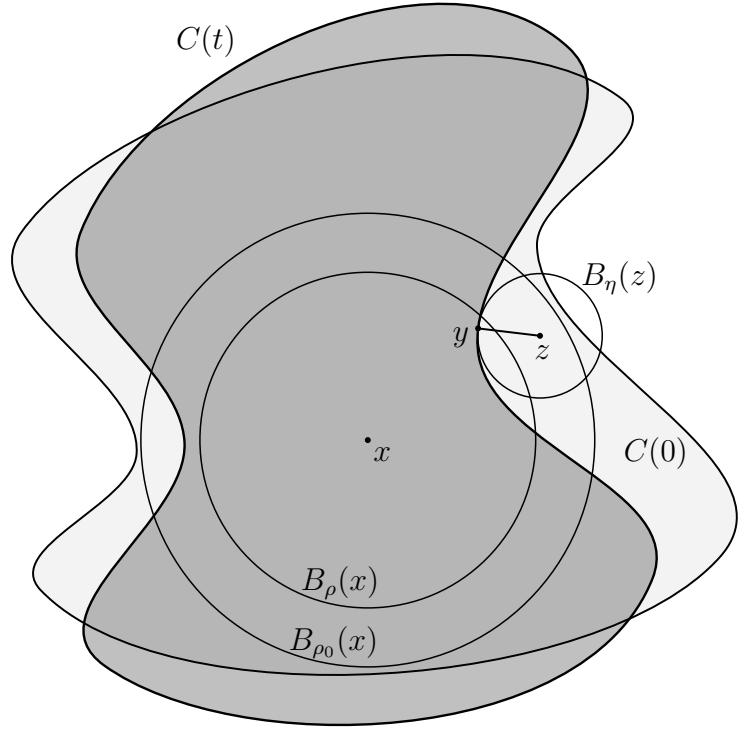
\begin{figure}[ht]
		\centering
\begingroup%
\makeatletter%
\providecommand\color[2][]{%
	\errmessage{(Inkscape) Color is used for the text in Inkscape, but the package 'color.sty' is not loaded}%
	\renewcommand\color[2][]{}%
}%
\providecommand\transparent[1]{%
	\errmessage{(Inkscape) Transparency is used (non-zero) for the text in Inkscape, but the package 'transparent.sty' is not loaded}%
	\renewcommand\transparent[1]{}%
}%
\providecommand\rotatebox[2]{#2}%
\newcommand*\fsize{\dimexpr\f@size pt\relax}%
\newcommand*\lineheight[1]{\fontsize{\fsize}{#1\fsize}\selectfont}%
\ifx\svgwidth\undefined%
	\setlength{\unitlength}{273.84028397bp}%
	\ifx\svgscale\undefined%
		\relax%
	\else%
		\setlength{\unitlength}{\unitlength * \real{\svgscale}}%
	\fi%
\else%
	\setlength{\unitlength}{\svgwidth}%
\fi%
\global\let\svgwidth\undefined%
\global\let\svgscale\undefined%
\makeatother%
\begin{picture}(1,0.99601562)%
	\lineheight{1}%
	\setlength\tabcolsep{0pt}%
	\put(0,0){\includegraphics[width=\unitlength,page=1]{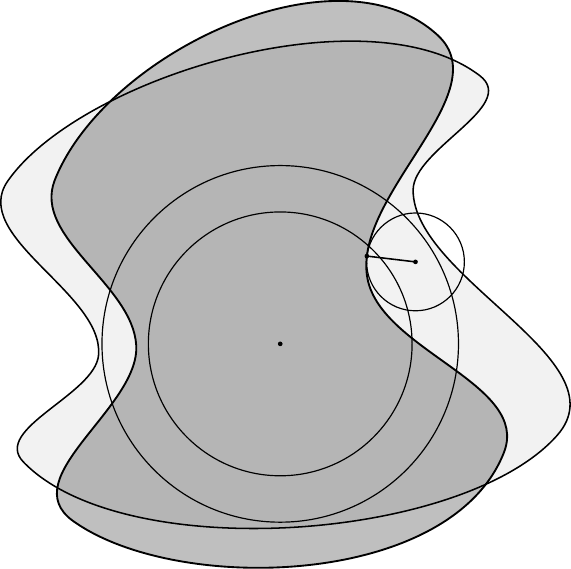}}%
	\put(0.84326142,0.36580422){\makebox(0,0)[lt]{\lineheight{1.25}\smash{\begin{tabular}[t]{l}$C(0)$\end{tabular}}}}%
	\put(0.22700628,0.93256445){\makebox(0,0)[lt]{\lineheight{1.25}\smash{\begin{tabular}[t]{l}$C(t)$\end{tabular}}}}%
	\put(0.49874634,0.36562787){\makebox(0,0)[lt]{\lineheight{1.25}\smash{\begin{tabular}[t]{l}$x$\end{tabular}}}}%
	\put(0.60806448,0.53112465){\makebox(0,0)[lt]{\lineheight{1.25}\smash{\begin{tabular}[t]{l}$y$\end{tabular}}}}%
	\put(0.72052908,0.50606554){\makebox(0,0)[lt]{\lineheight{1.25}\smash{\begin{tabular}[t]{l}$z$\end{tabular}}}}%
	\put(0.38240247,0.10530114){\makebox(0,0)[lt]{\lineheight{1.25}\smash{\begin{tabular}[t]{l}$B_{\rho_0}(x)$\end{tabular}}}}%
	\put(0.39892885,0.18699939){\makebox(0,0)[lt]{\lineheight{1.25}\smash{\begin{tabular}[t]{l}$B_\rho(x)$\end{tabular}}}}%
	\put(0.78437346,0.60742162){\makebox(0,0)[lt]{\lineheight{1.25}\smash{\begin{tabular}[t]{l}$B_\eta(z)$\end{tabular}}}}%
\end{picture}%
\endgroup%

		\caption{Visual aid for the proof of Lemma~\ref{Lunifball}.}
		\label{fig:inner-ball}
	\end{figure}
	Figure \ref{fig:inner-ball} illustrates this situation.
	From
	\[
		\norm{z-x}{}
		\leq \norm{z-y}{}+\norm{y-x}{}
		< \eta + \rho \leq \rho_0-\rho+\rho = \rho_0
	\]
	and from \ref{it:ball} we deduce $z\in B_{\rho_0}(x)\subseteq \C(t_0)$; hence by \ref{it:exc}, the definition of $\tau$,
	and $\C(t)\subseteq \H \setmeno B_\eta(z)$ we have
	\[
		\eta > \omega(\tau) \geq \omega(t)
		\geq \e\bigl(\C(t_0),\C(t)\bigr) \geq \d\bigl(z,\C(t)\bigr)
		\geq \d\bigl(z,\H \setmeno B_\eta(z)\bigr) \geq \eta,
	\]
	which is impossible. This concludes the proof of the claim.

	Next we show that we have $x\in \C(t)$ for every $t\in[t_0,\min\{t_0+\tau,T\}]$.
	Indeed, if it were not the case, we would have some $t\in[t_0,\min\{t_0+\tau,T\}]$ such that $x\not\in \C(t)$.
	From \ref{it:exc}, the definition or $\tau$, and recalling that $x\in \C(t_0)$ by \ref{it:ball} we get
	\[
		\d\bigl(x,\C(t)\bigr) \leq \e\bigl(\C(t_0),\C(t)\bigr) \leq \omega(t) \leq \omega(\tau) < \rho,
	\]
	therefore $\d\bigl(x,\partial \C(t)\bigr) = \d\bigl(x,\C(t)\bigr) < \rho$,
	which is impossible because it violates the first claim.

	Finally, from $x\in \C(t)$ and $\d\bigl(x,\partial \C(t)\bigr)\geq\rho$ for every $t\in[t_0,\min\{t_0+\tau,T\}]$
	we deduce that $B_\rho(x)\subseteq \C(t)$ for every $t\in[t_0,\min\{t_0+\tau,T\}]$, and this finishes the proof.
\end{proof}

We are now ready to establish uniform estimates for the variation in the case when $\C(t)$ enjoys the uniform inner cone condition.

\begin{Thm}\label{TV con cond}
	For every $t \in \clint{0,T}$ let $\C(t)$ be an $r$-prox-regular subset of $\H$, and assume that \eqref{e-cont->e-unifcont} holds. Moreover assume that there exist $R > 0$ and $d > 0$ such that for every $t \in \clint{0,T}$
	\begin{equation}\label{unif. int. cone}
		\forall x \in \partial\C(t) \ \exists z \in \C(t) \cap B_d(x) \ :\
		(1-\lambda)x+\lambda B_R(z) \subseteq \C(t), \ \forall \lambda \in \clint{0,1},
	\end{equation}

	Let $(\eps_n)_{n \in \en}$ and $(\delta_n)_{n\in \en}$ be sequences in $\ar_+$
	satisfying \eqref{epsn-1}--\eqref{CforP},
	let $(\mathfrak{S}_n)_{n \in \en}$ be a sequence of subdivisions of $\clint{0,T}$
	satisfying \eqref{subd1}--\eqref{subd4},
	let $(y_j^n)_{j =0}^{J^n}$ be the family of points in $\H$ defined by
	\eqref{implicit scheme for Lip-sweeping processes1}--\eqref{implicit scheme for Lip-sweeping processes3}
	by virtue of Lemma~\ref{catchup},
	and let $y_n : \clint{0,T} \function \H$
	be the discrete solution of Problem \ref{Pb}, i.e.\ the right continuous step function
	defined by \eqref{yn(0)}--\eqref{yn} in Lemma~\ref{exact discr.sol}.
	Then
	\begin{equation}
		\sup_{n \in \en}\V(y_n,\clint{0,T}) < \infty.
	\end{equation}
\end{Thm}

\begin{proof}
	Let us first choose $\lambda \in \opint{0,1}$ such that $(\lambda d+\lambda R/2)^2<2r(\lambda R/2)$,
	i.e. $\lambda(d+R/2)^2 < rR$.
	Now let $\tau = \tau(\omega,r,\lambda R,\lambda R/2)$ be the positive number obtained by applying
	Lemma~\ref{Lunifball} with $\omega(\delta) = \sup\{\e(\C(s),\C(t))\ :\ 0 \le t-s \le \delta\}$ being
	the continuity modulus of $\C$, $\rho_0=\lambda R$ and $\rho=\lambda R/2$.
	Let $\bar n\in\en$ be such that $\eps_{\bar n} < r/2$, $\delta_{\bar n}<\tau/2$ and
	$(\lambda d+\lambda R/2+\eps_{\bar n})^2<2r(\lambda R/2)$.

	Fix $n\in\en$ with $n\geq\bar n$. Given $i\in\{0,\dots,J^n-1\}$, consider the point $y^n_i\in\C(t^n_i)$;
	there are two cases: either $y^n_i\in\partial\C(t^n_i)$, or $y^n_i\in\Int(\C(t^n_i))$.
	\begin{itemize}
		\item Assume $y^n_i\in\partial\C(t^n_i)$.
		      Since the steps of the partition are shorter than $\delta_{\bar n}<\tau/2$, we can find an index
		      $j\in\{i+1,\dots,J^n\}$ such that $t^n_j \in \clint{\min\{t^n_i+\tau/2,T\},t^n_i+\tau}$.
		      By the assumption \eqref{unif. int. cone} there is $z^n_i\in\C(t^n_i)\cap B_d(y^n_i)$ such that,
		      letting $w_i^n=(1-\lambda)y^n_i+\lambda z^n_i$, we have
		      $B_{\lambda R}(w_i^n)\subseteq\C(t^n_i)$,
		      hence from condition \eqref{e-cont->e-unifcont}, Lemma~\ref{Lunifball}, and our choice of $\tau$,
		      we deduce that $B_{\lambda R/2}(w_i^n)\subseteq\C(t)$ for every $t\in\clint{t^n_i,t^n_j}$.
		      Applying Lemma~\ref{L:TV-est} with
		      $\C_k = \C(t_k^n)$ for $k=i, \ldots, j$, $y_0 = y_i^n$, $w = w_i^n$, $\rho = \lambda R /2$, and $\alpha=\lambda d+\lambda R/2+\eps_{\bar n},$ we deduce
		      \[
			      \V\bigl(y_n,\clint{t^n_i,t^n_j}\bigr) \leq r \frac{d^2-(\lambda R/2)^2}{\lambda rR-\alpha^2}.
		      \]
		      Observe that the numerator is positive because
		      $\lambda R/2=\rho\leq\lVert{w^n_i-y^n_i}\rVert\leq\lVert{z^n_i-y^n_i}\rVert\leq d$
		      and the denominator is positive too because by the choice of $\lambda$ and $\bar n$
		      we have $\alpha^2<2r(\lambda R/2)=\lambda rR$; this is the reason why in the previous
		      estimate of the variation we discarded the $\max$ resulting from \eqref{f-TV-est}.
		\item If $y^n_i\in\Int(\C(t^n_i))$, then there are two possible situations: either $y^n_j=y^n_i$
		      for every $j\in\{i,\dots,J^n\}$, in which case $\V(y_n,\clint{t^n_i,T})=0$, or there exists
		      $\bar\imath\in\{i+1,\dots,J^n\}$ such that $y^n_{\bar\imath}\neq y^n_i$ and $y^n_k=y^n_i$
		      for every $k\in\{i,\dots,\bar\imath-1\}$.

		      In this latter situation we have
		      $\norm{y^n_{\bar\imath}-y^n_{\bar\imath-1}}{}
			      \leq \e\bigl(\C(t^n_{\bar\imath-1}),\C(t^n_{\bar\imath})\bigr)
			      \leq \eps_n \leq \eps_{\bar n}$,
		      from which $\V(y_n,\clint{t^n_i,t^n_{\bar\imath}})\leq\eps_{\bar n}$.
		      The point $y^n_{\bar\imath-1}$ cannot belong to $\C(t^n_{\bar\imath})$, otherwise we would have
		      $y^n_{\bar\imath} = \P_{\C(t^n_{\bar\imath})}(y^n_{\bar\imath-1}) = y^n_{\bar\imath-1}$.
		      From this we deduce that $y^n_{\bar\imath}\in\partial\C(t^n_{\bar\imath})$, because if it were
		      in the interior it could not be the point minimizing the distance from $y^n_{\bar\imath-1}$.

		      We can now replicate the argument of the former case and deduce that there exists an index
		      $j\in\{\bar\imath+1,\dots,J^n\}$ such that
		      $t^n_j\in\clint{\min\{t^n_{\bar\imath}+\tau/2,T\},t^n_{\bar\imath}+\tau}$ and
		      \[
			      \V\bigl(y_n,\clint{t^n_{\bar\imath},t^n_j}\bigr) \leq r \frac{d^2-(\lambda R/2)^2}{\lambda rR-\alpha^2}.
		      \]
	\end{itemize}
	In both cases we were able to find $t^n_j\geq\min\{t^n_i+\tau/2,T\}$ such that
	\[
		\V\bigl(y_n,\clint{t^n_i,t^n_j}\bigr)
		\leq r \frac{d^2-(\lambda R/2)^2}{\lambda rR-\alpha^2} + \eps_{\bar n},
	\]
	therefore it takes at most $\lceil T /(\tau/2)\rceil$ steps to cover the entire interval
	$\clint{0,T}$ and we find
	\[
		\V(y_n,\clint{0,T})
		\leq \lceil2T/\tau\rceil \left(r \frac{d^2-(\lambda R/2)^2}{\lambda rR-\alpha^2} + \eps_{\bar n}\right).
		\qedhere
	\]
\end{proof}


\subsection{Limit procedure}

In this final subsection we perform the limit procedure needed to recover a solution of problem \ref{Pb}.
We start with the following convergence lemma on differential measures.

\begin{Lem}
	If $g_n, g, h \in BV(\clint{0,T};\H)$ for every $n \in \en$, $h$ is right-continuous, $g_n \to g$ uniformly on $\clint{0,T}$,
	and $\sup_{n \in \en}\V(g_n,\clint{0,T})$ $<$ $\infty$, then
	\begin{equation}\label{gn->g debst}
		\lim_{n \to \infty}\int_{\clint{0,T}} \duality{h(t)}{\de\D g_n(t)} = \int_{\clint{0,T}} \duality{h(t)}{\de\D g(t)}.
	\end{equation}
\end{Lem}

\begin{proof}
	Let us observe that we can assume without loss of generality that $g = 0$ and that $h$ and every $g_n$ are continuous at $T$. Set $V :=  \sup_{n \in \mathbb{N}} \mathrm{V}(g_n, [0,T]) < \infty$, and let $\eps > 0$ be arbitrarily fixed. Let $h_\eps = \sum_{j=1}^{m_\eps} h_\eps^j \indicator_{[t_{j-1},t_j[}$  be a step map with respect to the algebra of subintervals of $[0,T]$ such that $\|h_\eps - h\|_\infty < \eps(2V)^{-1}$, and let  $n_\eps \in \mathbb{N}$ be such that
	$\|g_n\|_\infty < \eps(4m_\eps  \|h_\eps\|_\infty)^{-1} $ for every $n \ge n_\eps$. Then for every $n \ge n_\eps$ we have that
	\begin{align}
		\left|\int_{[0,T]}h\de \D g_n \right|
		 & \le \int_{[0,T]}|h - h_\eps|\de |\D g_n| + \left|\int_{[0,T]}h_\eps\de \D g_n\right| \notag \\
		 & \le \|h_\eps - h\|_\infty \sup_{n \in \mathbb{N}_+} \mathrm{V}(g_n, [0,T]) +
		\sum_{j=1}^{m_\eps} \norm{h_\eps^j}{} \norm{g_n(t_j) - g_n(t_{j-1}-)}{}\notag                  \\
		 & \le \eps + 2m_\eps \|h_\eps\|_\infty \|g_n\|_\infty < \eps. \notag
	\end{align}
\end{proof}

We can show the limit procedure.

\begin{Thm}\label{limit}
	For every $t \in \clint{0,T}$ let $\C(t)$ be an $r$-prox-regular subset of $\H$, and assume that \eqref{e-cont->e-unifcont} holds. Fix $y_0 \in \C(0)$.
	Let  $(\eps_n)_{n\in \en}$ be a sequence in $\ar$ satisfying \eqref{epsn-2}. Let
	$(\delta_n)_{n\in \en}$ be a sequence in $\ar$, and
	$(\mathfrak{S}_n)_{n \in \en}$ be a sequence of subdivisions of $\clint{0,T}$, satisfying \eqref{deltan}--\eqref{CforP}, and \eqref{subd1}--\eqref{subd4}, and let $(y_j^n)_{j =0}^{J^n}$ is the family in
	$\H$ defined in
	\eqref{implicit scheme for Lip-sweeping processes1}--\eqref{implicit scheme for Lip-sweeping processes3} by virtue of Lemma \ref{catchup}. For every $n \in \en$, we let $y_n : \clint{0,T} \function \H$
	be the discrete solution of Problem \ref{Pb}, i.e. the right continuous step function defined by \eqref{yn(0)}--\eqref{yn},
	assume that $\sup_{n \in \en}\V(y_m,\clint{0,T}) < \infty$, and let $y \in \BV(\clint{0,T};\H) \cap \Czero(\clint{0,T};\H)$ be the uniform limit of $y_n$, which exists by Theorem
	\ref{T:unif conv}. Then $y$ is a solution of Problem \ref{Pb}.
\end{Thm}

\begin{proof}
	Let $(y_n)$ be the sequence defined in \eqref{yn(0)}-\eqref{yn} of Theorem \ref{T:unif conv} and let $y \in \BV(\clint{0,T};\H) \cap \Czero(\clint{0,T};\H)$ be its uniform limit.
	From
	\eqref{implicit scheme for Lip-sweeping processes3} it follows that at every step $n$ we have that
	\begin{equation}\label{Dy_j}
		-\D y_n(\{t_j^n\}) = y_j^n - y_{j-1}^n \in N_{\C(t_j^n)}(y_j^n) = N_{\C(t_j^n)}(y_n(t_j^n)),
	\end{equation}
	and we also trivially have that
	\begin{equation}\label{DI_j}
		\D y_n(\opint{t_{j-1}^n,t_j^n}) = 0 \in N_{\C(t)}(y_n(s)) \qquad
		\forall s \in \opint{t_{j-1}^n,t_j^n}
	\end{equation}
	for every $j$.
	Let also consider the continuous interpolation $x_n$ defined in \eqref{x_n}.
	Now fix $\tau \in \clsxint{0,T}$ and $z_\tau \in \C(\tau)$ arbitrarily. By applying Theorem \ref{T:unif conv} with $y_0$ and $\C(t)$ replaced respectively by $z_\tau$ and $\C(\tau + (T-\tau)t/T)$, we infer the existence of a continuous function $\widetilde{z} : [0,T] \function \H$ such that $\widetilde{z}(0) = z_\tau$
	and $\widetilde{z}(t) \in \C(t)$ for all $t \in \clint{0,T}$. Therefore if we define $z : \clint{\tau,T} \function \H$ by
	\[
		z(s) := \widetilde{z}\left(\frac{T}{T-t}(s-\tau)\right) \qquad s \in \clint{\tau,T},
	\]
	then $z$ is continuous, $z(\tau) = z_\tau$, and $z(t) \in \C(t)$ for every $t \in \clint{t,T}$. Therefore from \eqref{Dy_j}, \eqref{DI_j}
	we infer that
	\begin{equation}
		\duality{-\D y_n(\{s\})}{z_s - y_n(s)} \le \frac{1}{2r}\norm{\D y_n(s)}{}\norm{z_s - y_n(s)}{}^2 \qquad
		\forall z_s \in \C(s),\   \forall s \in \clint{0,T}.
	\end{equation}
	It follows that for every $h > 0$ small enough we have
	\begin{equation}\label{discr int form}
		\int_{\clint{t,t+h}} \duality{y_n(s) - z(s)}{\de \D y_n(s)} \le
		\frac{1}{2r}\int_{\clint{t,t+h}} \norm{z(s) - y_n(s)}{}^2\de \vartot{\D y_n}(s)
	\end{equation}
	By the boundedness of $V(y_n,\clint{0,T})$ we have that $\vartot{\D y_n}(\clint{0,T})$ is  bounded, hence there exists a positive Borel measure $\mu : \borel(\clint{0,T}) \function \clsxint{0,\infty}$ such that $\vartot{\D y_n} \convergedebstar \mu$
	at least for a subsequence which we dot relabel. Therefore taking the limit in \eqref{discr int form} as $n \to \infty$, by the continuity of the integrands, and using the fact that $x_n - y_n \to 0$ uniformly, we infer that
	\begin{equation}
		\int_{\clint{t,t+h}} \duality{y(s) - z(s)}{\de \D y(s)} \le \frac{1}{2r}\int_{\clint{t,t+h}} \norm{z(s) - y(s)}{}^2\de \mu(s)
	\end{equation}
	For every $k \in \en$ set $I_k := \{t \in I : \d(t,\partial I) > 1/k\}$ and let $\phi_k \in \Czero(\clint{0,T};\ar)$ be such that
	$\indicator_{I_k} \le \phi_k \le \indicator_I$. Then by the lower semicontinuity of the pointwise variation w.r.t.\ the pointwise convergence we have
	\[\begin{split}
			\vartot{\D y}(I_k)
			&= \V(y,I_k) \le \liminf_{n \to \infty}\V(y_k,I_k)
			= \liminf_{n \to \infty} \vartot{\D y_n}(I_k)               \\
			&\le  \liminf_{n \to \infty} \int_{\clint{0,T}} \phi_k \de\vartot{\D y_n}
			= \int_{\clint{0,T}} \phi_k \de\mu    \\
			&\le \int_{\clint{0,T}} \indicator_I\de\mu
			= \mu(I),
		\end{split}\]
	therefore taking the limit as $k \to \infty$ we infer that $\vartot{\D y}(I) \le \mu(I)$. Since every open set $A \subseteq \clint{0,T}$ is a disjoint union of open intervals, we infer that $\vartot{\D y}(A) \le \mu(A)$. This inequality and the outer regularity of the two measures finally yield $\vartot{\D y}(B) \le \mu(B)$ for every Borel set $B$, hence $\D y$ is absolutely continuous with respect to $\mu$.

	Thus there exists  $v \in \L^1(\mu;\H)$ such that $\D y = v\mu$, hence
	\begin{equation}
		\int_{\clint{t,t+h}} \duality{y(s) - z(s)}{v(s)}\de \mu(s) \le \frac{1}{2r}\int_{\clint{t,t+h}} \norm{z(s) - y(s)}{}^2\de \mu(s).
	\end{equation}
	If $\tau$ is a $\mu$-Lebesgue point of $s \longmapsto \duality{y(s) - z(s)}{v(s)}$, dividing both sides by $\mu(\clint{\tau,\tau+h})$ and taking the limit as $h \searrow 0$ we infer that
	\[
		v(\tau) \in -N_{\C(\tau)}(y(\tau)).
	\]
	The conclusion follows from Proposition \ref{propsigma} and from the fact that the set of Lebesgue points of
	$s \longmapsto \duality{y(s) - z(s)}{v(s)}$ has full $\mu$-measure in $\clint{0,T}$. Indeed if $L$ is the set of $\mu$-Lebesgue points for $\tau \longmapsto v(\tau)$, then $L$ has full measure, thus fix $t \in L$, and choose $\zeta_t \in \C(t)$ arbitrarily. Since
	$y \in \Czero(\clint{0,T};\H)$ we have that $t$ is a $\mu$-Lebesgue points of $y$, therefore
	\begin{align}
		 & \frac{1}{\mu(\clint{t,t+h})} \int_{\clint{t,t+h}} \left|\duality{y(\tau) - z(\tau)}{v(\tau)} - \duality{y(t) - z(t)}{v(t)}\right| \de \tau
		\notag                                                                                                                                        \\
		 & \le
		\frac{1}{\mu(\clint{t,t+h})} \int_{\clint{t,t+h}}
		\left(\norm{y-z}{\infty} \norm{v(\tau) - v(t)}{} + \norm{v(t)}{} \norm{y(\tau) - y(t)}{}\right) \de \tau, \label{pto leb 1}
	\end{align}
	thus
	\begin{equation}
		\lim_{h \searrow 0} \frac{1}{\mu(\clint{t,t+h})} \int_{\clint{t,t+h}} \duality{y(\tau) - z(\tau)}{v(\tau)}  \de \mu(\tau)
		= \duality{y(t) - z(t)}{v(t)}. \label{pto leb 2}
	\end{equation}
\end{proof}


\subsection{Conclusion}

In this subsection we prove the main theorems \ref{main th1} and \ref{main th2} of the paper by collecting together the results in the previous subsections. Indeed let $(y^n_j)_{j=0}^{J^n}$ be the sequence of discrete solution obtained by the catching-up algorithm in Lemma \ref{catchup}. Moreover let $(y_n)$ be the sequence of discrete solutions of Problem \ref{Pb}, i.e. the sequence of piecewise constant approximated solutions defined by \eqref{yn(0)}--\eqref{yn}.

Under the assumptions of Theorem \ref{main th2}, we have that $\sup_{n \in \en}\V(y_n,\clint{0,T})$ is finite
by virtue of Lemma \ref{TVest-ball}.
Instead under the assumptions of Theorem \ref{main th1}, the boundedness of $\V(y_n,\clint{0,T})$ is deduced
by exploiting Theorem \ref{TV con cond}.

In both situation we can therefore apply Theorem \ref{T:unif conv} and Theorem \ref{limit}, and infer the existence of a solution
of Problem \ref{Pb}. Uniqueness is classic and follows for instance from \cite[Proposition 3.6]{Thi16}.


\end{document}